\documentclass[11pt, a4paper]{article}

\usepackage{inputenc} 
\usepackage[english]{babel}
\usepackage{dsfont}
\usepackage{enumerate}

\usepackage[totalheight=24 true cm, totalwidth=17 true cm]{geometry}
\usepackage{enumitem}

\usepackage{bbding}
\usepackage{amsmath,multirow}
\usepackage{amsthm}
\usepackage{amssymb}
\usepackage{subcaption}
\usepackage{mathrsfs}
\usepackage{mathtools}
\usepackage{ stmaryrd }
\usepackage{url}
\usepackage{wrapfig}
\usepackage{starfont}
\usepackage{pifont}
\usepackage{eurosym}

\usepackage{imakeidx}
\makeindex[]

\usepackage{multicol}
\usepackage{relsize}
\usepackage{float}
\usepackage{tikz,pgf}
\usepackage{tikz-cd}
\usepackage{wasysym}
\usepackage{graphicx}
\usepackage{fancyhdr}
\usepackage{verbatim}

\usepackage{pgfplots}
\pgfplotsset{compat=1.15}
\usepackage{mathrsfs}
\usetikzlibrary{arrows}

\usepackage[all,dvips,knot,web,arc,curve,color,frame]{xy}
\usepackage[all,knot,arc,color,web]{xy}
\xyoption{arc}
\xyoption{web}
\xyoption{curve}

\newcommand{\Ccal}{\mathcal{C}}
\newcommand{\fn}{\mathbf{n}}
\newcommand{\fl}{\mathbf{\lambda}}

\numberwithin{equation}{section}

\theoremstyle{plain}
\newtheorem{theorem}{Theorem}[section]
\newtheorem{proposition}[theorem]{Proposition}
\newtheorem{corollary}[theorem]{Corollary}
\newtheorem{lemma}[theorem]{Lemma}

\theoremstyle{definition}
\newtheorem{definition}[theorem]{Definition}
\newtheorem{setup}[theorem]{Setup}

\newtheorem{remark}[theorem]{Remark}

\newtheorem{notation}[theorem]{Notation}
\newtheorem{example}[theorem]{Example}
\newtheorem{question}{Question}
\usepackage[bookmarksnumbered=true]{hyperref} 
\hypersetup{
     colorlinks = true,
     linkcolor = blue,
     anchorcolor = blue,
     citecolor = teal,
     filecolor = blue,
     urlcolor = blue
     }
\newcommand\restr[2]{{% we make the whole thing an ordinary symbol
  \left.\kern-\nulldelimiterspace % automatically resize the bar with \right
  #1 % the function
  \vphantom{\small|} % pretend it's a little taller at normal size
  \right|_{#2} % this is the delimiter
  }}

\newcommand{\vs}{\vspace*{.5em}}

\everymath{\displaystyle}
\allowdisplaybreaks

\makeatletter
\def\mathcenterto#1#2{\mathclap{\phantom{#1}\mathclap{#2}}\phantom{#1}}
\let\old@widetilde\widetilde
\def\widetildeto#1#2{\mathcenterto{#2}{\old@widetilde{\mathcenterto{#1}{#2\,}}}}
\let\old@widehat\widehat
\def\widehatto#1#2{\mathcenterto{#2}{\old@widehat{\mathcenterto{#1}{#2\,}}}}
\makeatother

\newtheorem{theoremA}{Theorem}

\newtheorem{theoremB}{Theorem}

\newtheorem{theoremC}{Theorem}

\newtheorem{theoremD}{Theorem}

\newcommand{\is}[1]{\langle  #1 \rangle} % \ip a b le pone < >
\newcommand{\abs}[1]{\left| #1 \right|} % \abs a b le pone | |

\newcommand{\size}[1]{\left| #1 \right|} % \size a b le pone | |
\newcommand{\pare}[1]{\left( #1 \right)} % \pare a b le pone ( )
\newcommand{\set}[1]{{\left\{ #1 \right\}}} % \set a b le pone { }

\newcommand{\corch}[1]{\left[ #1 \right]} % \pare a b le pone [ ]

\newcommand*\closure[1]{\overline{#1}}

\DeclareMathOperator{\rank}{rk}

\DeclareMathOperator{\im}{im}
\def\dim{\operatorname{dim}}

\DeclareMathOperator{\CC}{\mathbb{C}}
\DeclareMathOperator{\ZZ}{\mathbb{Z}}

\hypersetup{
    colorlinks=true,
    linkcolor=blue,
    filecolor=magenta,      
    urlcolor=cyan,
    pdftitle={Overleaf Example},
    pdfpagemode=FullScreen,
    }
    
\setlist[itemize]{noitemsep}

\normalfont

\usepackage{bbm}
\usepackage{dsfont}

\usepackage{url}
\usepackage{amsmath,blkarray,booktabs, bigstrut}
\usepackage{tikz}
\usepackage{multirow}
\usepackage{xcolor}
\usepackage{color}
\usepackage[maxbibnames=99]{biblatex}
\title{%Hypergraph Varieties: Irreducible Decomposition and Applications to Conditional Independence Models
Decomposing Conditional Independence Ideals with Hidden Variables: A Matroid-Theoretic Approach}
\author{Emiliano Liwski}
\date{\today}

\begin{document}
\maketitle

\begin{abstract}
We study a class of determinantal ideals arising from conditional independence (CI) statements with hidden variables. Such CI statements translate into determinantal conditions on a matrix whose entries represent the probabilities of events involving the observed random variables. Our main objective is to determine the irreducible components of the corresponding varieties and to provide a combinatorial or geometric interpretation of each. We achieve this by introducing a new approach rooted in matroid theory. In particular, we introduce a new class of matroids, which we term quasi-paving matroids, and show that the components of these determinantal varieties are precisely matroid varieties of quasi-paving matroids. Moreover, we derive generating functions that encode the number of irreducible components of these CI ideals.

%We study hypergraph varieties and their irreducible decompositions, introducing the framework of quasi-paving matroids. This approach enables us to explicitly determine the irreducible decomposition of hypergraph varieties associated with a specific class of paving matroids. Utilizing these results, we provide a concrete decomposition for a notable family of hypergraph varieties arising in conditional independence (CI) models with hidden variables.

%We study %paving matroids 
%hypergraph varieties and their %matroid and circuit varieties
%irreducible decomposition, introducing the family of quasi-paving matroids in this framework. Through this new perspective, we obtain the irreducible decomposition of %circuit varieties
%hypergraph varieties associated with a specific class of paving matroids. Using this we explicetely find the irreducible decomposition of %two notable families of matroids, including 
%a notable class of hypergraph varieties related to conditional independence (CI) models with hidden variables. %Additionally, we offer simpler geometric proofs of the main results of \cite{clarke2020conditional}, which are also connected to these models.
\end{abstract}

{\hypersetup{linkcolor=black}
{\tableofcontents}}

%\vspace{50pt}

\section{Introduction}

\subsection{Motivation}\label{motivation}

This work focuses on conditional independence (CI) statements involving hidden variables and their connections to matroid theory. In particular, our aim is to compute the irredundant irreducible decomposition of CI ideals. We begin with a concise overview of both conditional independence ideals and matroids.

\medskip
\noindent
\textbf{\large Conditional independence ideals.}

\medskip

Conditional independence models play a central role in algebraic statistics \cite{Studeny05:Probabilistic_CI_structures}, as they capture the relationships among random variables by encoding both conditional dependencies and independencies. A fundamental question in this setting is to determine the set of distributions that satisfy the given CI constraints. The problem becomes significantly more intricate when some variables are unobserved or hidden, introducing additional layers of complexity. In particular, a key challenge is to identify the constraints on the observed variables that emerge from dependencies involving hidden ones \cite{Steudel-Ay}.

This leads naturally to an algebraic formulation: CI statements correspond to the solutions of certain systems of polynomial equations \cite{Sullivant}. These equations generate the so-called CI ideal, and the distributions compatible with the model correspond to the intersection of this ideal with the probability simplex. When no hidden variables are involved, these equations are binomials, and the corresponding ideals have been widely studied; see, for example, \cite{herzog2010binomial, Rauh}. However, the presence of hidden variables greatly increases the complexity of the ideal, typically producing polynomials of arbitrary degree that are difficult to analyze and compute \cite{pfister2019primary}. 

\medskip

Let $X$, $Y_{1}$, $Y_{2}$ be observed random variables, and $H_{1}$, $H_{2}$ be hidden random variables, each taking values in finite sets $\mathcal{X}$, $\mathcal{Y}_{1}$, $\mathcal{Y}_{2}$, $\mathcal{H}_{1}$, $\mathcal{H}_{2}$, respectively, with cardinalities $\lvert \mathcal{X} \rvert = n$, $\lvert \mathcal{Y}_{1} \rvert = k$, $\lvert \mathcal{Y}_{2} \rvert = l$, $\lvert \mathcal{H}_{1} \rvert = s-1$, and $\lvert \mathcal{H}_{2} \rvert = t-1$. We consider the collection of conditional independence (CI) statements 
\begin{equation}\label{model}
\mathcal{C} = \{X \mathrel{\perp\!\!\!\perp} Y_{1} \mid \{ Y_{2}, H_{1} \}, X \mathrel{\perp\!\!\!\perp} Y_{2} \mid \{ Y_{1}, H_{2} \}\}.
\end{equation}

Our objective is to investigate the conditional independence ideal associated with this model. We introduce this ideal in the following definition.

\begin{definition}[CI ideal]\label{setup}
We define the $k \times l$ matrix of integers
\begin{equation}\label{matri}
\mathcal{Y} = \begin{pmatrix}
1 & k+1 & \cdots & (l-1)k+1 \\
2 & k+2 & \cdots & (l-1)k+2 \\
\vdots & \vdots & \ddots & \vdots \\
k & 2k & \cdots & lk
\end{pmatrix}.
\end{equation}

For each $i \in [k]$ and $j \in [l]$, we denote the supports of the rows and columns of $\mathcal{Y}$ by
\begin{align*}
R_{i} &= \{ i, k+i, \ldots, (l-1)k + i \}, \\
C_{j} &= \{ (j-1)k + 1, \ldots, (j-1)k + k \}.
\end{align*}

We denote by $R=\{R_{1},\ldots,R_{k}\}$ and $C=\{C_{1},\ldots,C_{l}\}$ the set of rows and columns of $\mathcal{Y}$, respectively. Given integers $s \leq k$ and $t \leq l$, we define the collection of subsets of $[kl]$:
\[
\Delta^{s,t} = \bigcup_{i \in [k]} \binom{R_i}{t} \cup \bigcup_{j \in [l]} \binom{C_j}{s}.
\]

Fix a positive integer $n \geq \max\{s,t\}$, and let $X = (x_{i,j})$ be an $n \times kl$ matrix of indeterminates. We define:

\vspace*{-.5em}

\hspace*{-1cm}\parbox{\linewidth}{
\begin{itemize}[label=$\blacktriangleright$]
\item The \emph{determinantal hypergraph ideal} of $\Delta^{s,t}$ as
\[
I_{\Delta^{s,t}} = \left\langle \left[ A \mid B \right]_X : A \subset [n], \, B \in \Delta^{s,t}, \, |A| = |B| \right\rangle \subset \CC[X],
\]
where $\left[ A \mid B \right]_X$ denotes the minor of $X$ with rows indexed by $A$ and columns indexed by $B$.\vs
\item The associated variety, which is the zero set of $I_{\Delta^{s,t}}$, is
\[
V_{\Delta^{s,t}} = \left\{ Y \in \CC^{n \times kl} : \operatorname{rank}(Y_F) < |F| \ \text{for each } F \in \Delta^{s,t} \right\},
\]
where $Y_F$ is the submatrix of $Y$ with columns indexed by $F$.
\end{itemize}}
The ideal $I_{\Delta^{s,t}}$ coincides with the conditional independence ideal associated to the model described in~\eqref{model}; see \cite{clarke2020conditional,clarke2022conditional}.
\end{definition}

The main question addressed in this work is the following:

\begin{question}\label{question decomposition}
Determine the irredundant irreducible decomposition of the variety $V_{\Delta^{s,t}}$.
\end{question}

\begin{comment}
\begin{example}
Consider the case where $s=t=n$ and $k,l\geq n$. The joint distribution of $Y_{1}$ and $Y_{2}$ has state space $\mathcal{Y}=\mathcal{Y}_{1}\times \mathcal{Y}_{2}$, which can be identified with the $k\times l$ matrix $\mathcal{Y}$ whose entries take values in the set $[kl]$, as described in~\eqref{matri}. For this model, the variety corresponding to the CI ideal $I_{\Delta}$ is precisely the circuit variety of the $n$-paving matroid $G_{k,l}$ from Definition~\ref{matroid gkl}. The irreducible decomposition of this variety is described in Theorem~\ref{Theorem 2}.
\end{example}
\end{comment}

For large values of $n$, determining the irreducible decomposition of the associated variety $V_{\Delta^{s,t}}$ becomes a substantial challenge. Nonetheless, several works have addressed Question~\ref{question decomposition} in specific instances of the parameters $s, t, k, l, n$: 

\vspace*{-.5em}

\hspace*{-1cm}\parbox{\linewidth}{
\begin{itemize}[label=$\blacktriangleright$]
\item In \cite{clarke2020conditional}, the irreducible decomposition of $V_{\Delta^{s,t}}$ was obtained for the case $s=2$ and $t=l$.\vs
\item The case $t=3$ was studied in \cite{clarke2022conditional}.\vs
\item In \cite{pfister2019primary}, the authors developed a method to compute the primary decomposition of $I_{\Delta^{s,t}}$ in the case $s=t=k=n=3$ and $l=4$. However, their computations reach the limits of current computer algebra systems. More recently, \cite{Fatemeh3} provided the explicit irreducible decomposition of the variety corresponding to this ideal.\vs
\item Cases such as $t = l$ and $t = 2$ have been studied in \cite{clarke2021matroid}.\vs
\item The case $s=k=2$ and $n=t$ has been recently investigated in \cite{alexandr2025decomposing}.
\end{itemize}}

In this work, we focus on the case $s=t=n$. 

\begin{question}\label{question stn}
Determine the irredundant irreducible decomposition of the variety $V_{\Delta^{s,t}}$ when $s=t=n$.
\end{question}

We emphasize that current computer algebra systems are unable to resolve Question~\ref{question stn}. To address this problem, we adopt a matroid-theoretic perspective. Using this perspective, we obtain a partial answer for arbitrary $n$ and a complete, explicit answer when $n=3$.

\medskip
\noindent
\textbf{\large Matroids.}

\medskip

Matroids provide a combinatorial framework for capturing linear dependence in vector spaces~\cite{whitney1992abstract, Oxley, piff1970vector}. Given a finite collection of vectors, the linearly dependent subsets determine a matroid.

To address Question~\ref{question stn}, we adopt an approach rooted in matroid theory. In particular, we work with two varieties naturally associated to a matroid $M$: the \emph{matroid variety} $V_{M}$ and the \emph{circuit variety} $V_{\mathcal{C}(M)}$.
We next introduce the definition of these varieties. %We refer the reader to Subsection~\ref{matroids} for the relevant background on matroids.

\begin{definition}\normalfont\label{def: real}
Let $M$ be a matroid of rank $n$ on $[d]$. A \emph{realization} of $M$ is a tuple of vectors $\gamma=(\gamma_{1},\ldots,\gamma_{d})$ in $\CC^{n}$ such that
\[
\{i_{1},\ldots,i_{p}\}\ \text{is dependent in $M$} 
\;\Longleftrightarrow\; 
\{\gamma_{i_{1}},\ldots,\gamma_{i_{p}}\}\ \text{is linearly dependent}.
\]
Equivalently, $\gamma$ may be interpreted as a matrix in $\CC^{n\times d}$ whose columns are $\gamma_{1},\ldots,\gamma_{d}$. 

The \emph{realization space} of $M$ is
\[
\Gamma_{M} \;=\; \{\gamma \in \CC^{nd} \mid \text{the column matroid of $\gamma$ realizes $M$}\},
\]
where $\CC^{nd}$ has coordinate ring 
\[
S = \CC[x_{ij} \mid 1 \leq i \leq n,\, 1 \leq j \leq d].
\]
Thus, each element of $\Gamma_{M}$ corresponds to an $n \times d$ matrix over $\CC$, or equivalently, to a tuple of $d$ vectors in $\CC^n$.  
A matroid is called \emph{realizable} if its realization space is non-empty.  
The \emph{matroid variety} $V_{M}$ is defined as the Zariski closure of $\Gamma_{M}$ in $\CC^{nd}$.
\end{definition}

We emphasize that this notion of matroid variety differs from another common one, where matroid varieties are viewed as subsets of a Grassmannian.
%the set 
%$\Gamma_{M}=\{\gamma\subset \CC^{n}: \gamma \ \text{is a realization of $M$}\}$.
%The matroid variety $V_M$ is defined as the Zariski closure of $\Gamma_M$, endowed with a rich geometric structure. 
%Introduced in \cite{gelfand1987combinatorial}, matroid varieties have since been extensively studied~\cite{clarke2021matroid, Vakil, sidman2021geometric, Fatemeh4, sturmfels1989matroid, feher2012equivariant, knutson2013positroid}. In this work, we study the circuit variety $V_{\Ccal(M)}$ of $M$, defined in terms of its circuits, which are the minimal dependent sets of $M$.

Matroid varieties are notable geometric objects lying at the intersection of algebraic geometry and combinatorics. Introduced in \cite{gelfand1987combinatorial}, they have since been extensively studied; see, for instance, \cite{clarke2021matroid,Fatemeh4,LIWSKI2026102484,sidman2021geometric,liwski2025algebraic}. Defined as the Zariski closure of matroid realization spaces, matroid varieties provide a natural geometric extension of matroid theory. The Mnëv-Sturmfels Universality Theorem \cite{mnev1988universality} reveals their complexity, establishing that matroid varieties satisfy ``Murphy's Law in Algebraic Geometry,'' meaning that every singularity type of a semi-algebraic set can be realized within a matroid variety.

\begin{definition}\normalfont\label{cir}
Let $M$ be a matroid of rank $n$ on $[d]$. 
Consider the $n \times d$ matrix $X = (x_{i,j})$ of indeterminates. The \emph{circuit ideal} of $M$ is defined as
\[
I_{\Ccal(M)} = \{ [A|B]_X : B \in \Ccal(M),\ A \subset [n],\ \text{and } |A| = |B| \},
\]
where $[A|B]_X$ denotes the minor of $X$ with rows indexed by $A$ and columns indexed by $B$.  

A tuple of vectors $\gamma = (\gamma_1, \ldots, \gamma_d)$ in $\CC^n$ is said to \emph{include the dependencies} of $M$ if
\[
\{i_1, \ldots, i_k\} \text{ is a dependent set of $M$} \ \Longrightarrow \ \{\gamma_{i_1}, \ldots, \gamma_{i_k}\} \text{ is linearly dependent}.
\]

The \emph{circuit variety} of $M$ is defined as
\[
V_{\Ccal(M)} = V(I_{\Ccal(M)}) = \bigl\{\gamma = (\gamma_1, \ldots, \gamma_d) \in \CC^{nd} : \gamma \text{ includes the dependencies of $M$} \bigr\}.
\]

Its ambient space is $\CC^{nd}$, and its coordinate ring is
\[
S = \CC[x_{ij} \mid 1 \leq i \leq n,\, 1 \leq j \leq d].
\]
\end{definition}

\begin{comment}
\begin{definition}\label{first definition}
Let $M$ be a matroid of rank $n$ on the ground set $[d]$.
\begin{itemize}
\item A tuple of vectors $\gamma = (\gamma_1, \ldots, \gamma_d)$ in $\CC^n$ is called a \emph{realization} of $M$ if
\[
\{i_{1},\ldots,i_{p}\} \text{ is a dependent set of } M \quad \Longleftrightarrow \quad \{\gamma_{i_{1}},\ldots,\gamma_{i_{p}}\} \text{ is linearly dependent.}
\]
The set of all realizations of $M$ forms its \emph{realization space}, denoted by
\[
\Gamma_{M} = \big\{ \gamma = (\gamma_1, \ldots, \gamma_d)\in (\CC^{n})^{d} : \gamma \text{ is a realization of } M \big\} \subset \CC^{nd}.
\]
The \emph{matroid variety} associated to $M$ is the Zariski closure $V_M = \overline{\Gamma_M}$.
\item A collection $\gamma = \{\gamma_1, \ldots, \gamma_d\} \subset \CC^n$ is said to \emph{include the dependencies} of $M$ if
\[
\{i_{1},\ldots,i_{p}\} \text{ is a dependent set of } M \quad \Longrightarrow \quad \{\gamma_{i_{1}},\ldots,\gamma_{i_{p}}\} \text{ is linearly dependent.}
\]
The \emph{circuit variety} of $M$ is defined as
\[
V_{\mathcal{C}(M)} = \big\{ \gamma \subset \CC^n : \gamma \text{ includes the dependencies of } M \big\} \subset \CC^{nd}.
\]
\end{itemize}
\end{definition}
\end{comment}

Circuit varieties form a distinguished class of determinantal varieties constructed from the collection of circuits of a matroid. The ideals defining matroid and circuit varieties frequently appear as primary components in decompositions of several families of ideals, including conditional independence ideals \cite{clarke2021matroid,clarke2020conditional,clarke2022conditional,caines2022lattice, mohammadi2018prime, ene2013determinantal}, as well as determinantal ideals studied in commutative algebra \cite{bruns2003determinantal,clarke2020conditional,clarke2021matroid,herzog2010binomial,pfister2019primary}. This underlines the central role of matroid and circuit varieties in the study of Question~\ref{question stn}. Decomposing the circuit variety $V_{\Ccal(M)}$ is generally a challenging problem. 
This question was recently approached in~\cite{liwski2025efficient} using a combinatorial strategy based on an algorithm for computing the minimal matroids of $M$. 
However, this approach is not guaranteed to succeed, and in general, a complete decomposition of circuit varieties remains unknown.

\medskip
\noindent
\textbf{\large Connection between CI ideals and matroids}

\medskip

The following result from \cite{clarke2021matroid} establishes a direct link between the variety $V_{\Delta^{s,t}}$ and the notions of matroid and circuit varieties.

\begin{theorem}\textup{\cite[Theorem~7.5]{clarke2021matroid}}\label{theorem grid}
Let $s, t, k, l, n$ be positive integers satisfying $3 \le s \le t \le l$, $s \le k$, and $t \le n \le s+t-3$. Then 
\[
\mathcal{C} = \min\big( \Delta^{s,t} \cup \textstyle\binom{[kl]}{n+1} \big)
\]
defines the collection of circuits of a realizable matroid $M$ on $[kl]$ of rank $n$, where 
$\min$ denotes the inclusion-minimal subsets. Moreover, %$M$ is the unique minimal matroid associated to $\Delta^{s,t}$, and 
the corresponding varieties coincide:
\[
V_{\Delta^{s,t}} = V_{\mathcal{C}(M)}.
\]
\end{theorem}

We now specialize to the case of interest, where $s = t = n \le \min\{k,l\}$. Denote $\Delta := \Delta^{s,t}$. Since $n \le s + t - 3$ holds automatically, the hypotheses of Theorem~\ref{theorem grid} are satisfied. It follows that there exists a matroid $G_{k,l}$ with circuits
\[
\mathcal{C} = \min\Big( \Delta \cup \textstyle\binom{[kl]}{n+1} \Big)
\]
such that
\[
V_{\Delta} = V_{\mathcal{C}(G_{k,l})},
\]
see Definition~\ref{matroid gkl}. Consequently, Question~\ref{question stn} reduces to the following:

\begin{question}\label{question gkl}
Determine the irredundant irreducible decomposition of $V_{\mathcal{C}(G_{k,l})}$.
\end{question}

\medskip
\noindent
\textbf{\large Strategy outline and our results.}

\medskip

We now outline our strategy for answering Question~\ref{question gkl}. Remarkably, the matroid $G_{k,l}$ belongs to the class of paving matroids, a particularly well-behaved and extensively studied family. It is conjectured in \cite{mayhew2011asymptotic} that asymptotically almost all matroids are paving; see also \cite{lowrance2013properties, mayhew2011asymptotic, oxley1991ternary, brualdi1972foundations, Oxley, welsh2010matroid}.

\begin{definition}\normalfont \label{pav}
A matroid $M$ of rank $n$ is called a \emph{paving matroid} if every circuit of $M$ has cardinality either $n$ or $n+1$. A \emph{dependent hyperplane} of $M$ is a maximal dependent set in which every subset of $n$ elements forms a circuit.
\end{definition}

We now introduce a new subclass of paving matroids that will play a central role in our discussion.

\begin{definition}\label{2 paving}
A \emph{tame paving matroid} is a paving matroid in which any three dependent hyperplanes have empty intersection.
\end{definition}

In Lemma~\ref{gkl is 2 paving}, we establish that the matroid $G_{k,l}$ is a tame paving matroid. This observation naturally leads to a broader and more ambitious problem, whose resolution would provide an answer to both Question~\ref{question gkl} and, in turn, to Question~\ref{question stn}.

\begin{question}\label{question 2 paving}
Let $M$ be a tame paving matroid. Determine the irreducible components of $V_{\mathcal{C}(M)}$.
\end{question}

In this work, we introduce the class of quasi-paving matroids and employ it to solve Question~\ref{question 2 paving}. As a consequence, we also resolve Question~\ref{question stn}, which was our initial objective.

\medskip

\noindent
\textbf{\large Our contributions.} We now summarize the main result of this paper. In Section~\ref{sec 3}, we introduce the family of \emph{quasi-paving} matroids, defined via a hypergraph construction that extends the class of tame paving matroids.

\begin{theoremA}\normalfont
Let $\mathcal{H} = \{H_{1}, \ldots, H_{k}\}$ be a collection of subsets of $[d]$ such that any three subsets have empty intersection, and let $n \leq d$ be a fixed positive integer. Then the following collection of subsets defines the set of circuits of a matroid:

\begin{itemize}
\item[{\bf (1)}] Every subset of $[d]$ of cardinality $n-1$ that lies in the intersection of two sets in $\mathcal{H}$.
\item[{\bf (2)}] Every subset of $[d]$ of cardinality $n$ contained in some $H_i \in \mathcal{H}$, provided it does not contain any subset of Type~{\bf (1)}.
\item[{\bf (3)}] Every subset of $[d]$ of cardinality $n+1$ that contains neither a subset of Type~{\bf (1)} nor of Type~{\bf (2)}.
\end{itemize}

Matroids arising from this construction are called \emph{quasi-paving} matroids. \hfill{\textup{(Theorem~\ref{quasi})}}
\end{theoremA}

%For any partition $\mathcal{Q}$ of the dependent hyperplanes of a paving matroid $M$ without points of degree greater than two, there is an associated quasi-paving matroid, denoted $M(\mathcal{Q})$, see Definition~\ref{constr}. Additionally, in Definition~\ref{nice}, we introduce the notion of {\em nice} partitions of the set of dependent hyperplanes of a paving matroid. We find the irreducible decomposition of circuit varieties associated with paving matroids with no points of degree greater than two.

We determine the irreducible components of circuit varieties associated to tame paving matroids, thereby answering Question~\ref{question 2 paving}. We refer to Definitions~\ref{constr} and~\ref{nice} for the relevant notions.

\begin{theoremB}\normalfont\label{theorem B}
Let $M$ be a tame paving matroid. Then the circuit variety $V_{\mathcal{C}(M)}$ has the following irredundant irreducible decomposition:
\begin{equation*}
V_{\mathcal{C}(M)} = \bigcup_{\mathcal{Q}} V_{M(\mathcal{Q})}, \tag{Theorem~\ref{Theorem}}
\end{equation*}
where the union is taken over all nice partitions $\mathcal{Q}$ of the dependent hyperplanes of $M$.
\end{theoremB}

By applying Theorem~\ref{theorem B}, we implicitly obtain the irredundant irreducible decomposition of the variety $V_{\Delta^{s,t}}$ in the case $s=t=n$, thereby giving a partial resolution of Question~\ref{question stn} in this setting. 
We say ``partial'' because the decomposition in Theorem~\ref{theorem B} is not fully explicit: identifying all nice partitions remains a nontrivial task. 
However, for the case $s=t=n=3$ we succeed in providing a complete and explicit description. 
We refer to Definitions~\ref{setup} and~\ref{matroid gkl} for the relevant notation.

\begin{theoremC}\normalfont \label{teoC}
For $s = t = n=3$, the variety $V_{\Delta^{s,t}}$ has the following irredundant irreducible decomposition:
\[
V_{\Delta^{s,t}} = \bigcup_{\mathcal{Q}} V_{G_{k,l}(\mathcal{Q})},\tag{Theorem~\ref{Theorem 2}}
\]
where the union ranges over all partitions $\mathcal{Q} = \{Q_1, \ldots, Q_m\}$ of $R  \cup C$ satisfying the following conditions for each $i \in [m]$ with $|Q_i| \geq 2$:
\begin{itemize}
\item[\textup{(i)}] $|Q_i \cap R| \geq 3$, $|Q_i \cap C| \geq 3$, and $\max\{|Q_i \cap R|, |Q_i \cap C|\} \geq 4$,
\item[\textup{(ii)}] If $\lvert \mathcal{Q} \rvert > 1$, then neither $R \subseteq Q_i$ nor $C \subseteq Q_i$,
\end{itemize}
where $R$ and $C$ denote the row and column sets of the matrix $\mathcal{Y}$ from~\eqref{matri}.
Moreover, Theorem~\ref{theorem: count components 1} provides a generating function encoding the number of irreducible components of $V_{\Delta^{s,t}}$.
\end{theoremC}

Finally, we apply Theorem~\ref{theorem B} to a notable family of rank-three matroids. Specifically, we consider the matroids $L_{n}$, introduced in Definition~\ref{confi ln}, which are constructed from a configuration of $n$ lines $\mathcal{L}$ in $\mathbb{P}^{2}$ together with their $\textstyle \binom{n}{2}$ pairwise intersection points. We determine the irreducible decomposition of the circuit variety associated with $L_{n}$.

\begin{theoremD}\normalfont\label{teo D}
The circuit variety of $L_{n}$ has the following irredundant irreducible decomposition:
\[
V_{\mathcal{C}(L_{n})} = \bigcup_{\mathcal{Q}} V_{L_{n}(\mathcal{Q})},\tag{Theorem~\ref{decomp ln}}
\]
where the union ranges over all partitions $\mathcal{Q} = \{Q_1, \ldots, Q_k\}$ of $\mathcal{L}$ such that $|Q_i| \notin \{2, 3, n-1\}$ for every $i \in [k]$. Moreover, Theorem~\ref{thm: count components 2} provides a generating function encoding the number of irreducible components of $V_{\Ccal(L_{n})}$.
\end{theoremD}

We also note that, prior to this work, the only instances in which the decomposition of $V_{\Delta^{s,t}}$ with $s=t=n$ or of $V_{\mathcal{C}(L_n)}$ was explicitly known are the case $s=t=n=k=3$, $l=4$, and the case $n=4$, respectively. These correspond to the $3\times 4$ grid configuration and the quadrilateral set configuration studied in~\cite{Fatemeh3}. Thus, Theorems~\ref{teoC} and~\ref{teo D} provide a substantial generalization of these previously studied examples.

%In Subsection~\ref{subsec 5.2} we study the hypergraph variety $V_{\Delta^{s,t}}$, which emerges from the study of conditional independence (CI) models. These varieties are defined in terms of parameters $s,t,k,l$, and $n$, see Setup~\ref{setup}. In \cite{clarke2020conditional}, the irreducible decomposition of the variety $V_{\Delta^{s,t}}$ was determined for the case where $s=2$ and $t=l$. We provide a simpler geometric proof of this result. Recall Definition~\ref{hypergraphs}.

%In Setup~\ref{setup} we introduce the hypergraph variety $V_{\Delta^{s,t}}$, which emerges from the study of conditional independence (CI) models. These varieties are defined in terms of parameters $s,t,k,l$, and $n$. 

\medskip
\noindent
\textbf{Outline.} In Section~\ref{sec 2}, we introduce the necessary background on matroids. Section~\ref{sec 3} is devoted to the introduction of the family of \emph{quasi-paving} matroids, along with a discussion of their main properties. Building on this framework, Section~\ref{sec 4} establishes the irredundant irreducible decomposition of circuit varieties associated with tame paving matroids. In Section~\ref{sec 5}, we apply these results to determine the irreducible components of $V_{\Delta^{s,t}}$ in the case $s=t=n=3$. Section~\ref{examples} focuses on a family of rank-three matroids $L_{n}$ arising from configurations of $n$ lines and their pairwise intersection points, for which we describe the irreducible components of their circuit varieties. In Section~\ref{section: number of irreducible components}, we derive generating functions that count the number of irreducible components of $V_{\Delta^{s,t}}$ and $V_{\Ccal(L_{n})}$. Finally, Section~\ref{appendix} contains the technical proof of Theorem~\ref{Theorem}.

\smallskip

\noindent{\bf Acknowledgement.}  
The author would like to thank Fatemeh Mohammadi for helpful conversations and for bringing this problem to his attention. The author would also like to thank Sebastian Seemann for helpful comments and suggestions regarding generating functions of integer and vector partitions.
This work was supported by the PhD fellowship 1126125N and partially funded by the FWO grants G0F5921N (Odysseus), G023721N, and the KU Leuven grant iBOF/23/064.

\section{Preliminaries}\label{sec 2}

Throughout this work, we write $[n] = \{1, \ldots, n\}$, and denote by $\textstyle \binom{C}{n}$ the collection of $n$-element subsets of a set $C$.

%We begin by establishing notation used throughout the note. Let $[n]$ denote the set $\{1,\ldots,n\}$. For an $n\times d$ matrix $X=\pare{x_{i,j}}$ of indeterminates, we denote by $\CC[X]$ the polynomial ring in the variables $x_{ij}$. For subsets $A\subset[n]$ and $B\subset[d]$ with $\size{A}=\size{B}$, we use $[A|B]_{X}$ to denote the minor of $X$ formed by rows indexed by $A$ and columns indexed by $B$. Additionally, let $X_{F}$ denote the submatrix of $X$ with columns indexed by $F\subset [d]$.

\subsection{Matroids} \label{matroids}

We begin by briefly recalling fundamental notions about matroids, and refer the reader to \cite{Oxley} for further details.  Recall first the definition of a matroid in terms of its independent sets.

\begin{definition}
A \emph{matroid} $M$ consists of a ground set $[d]$ together with a collection $\mathcal{I}(M)$ of subsets of $[d]$, called \emph{independent sets}, satisfying: 
\begin{enumerate}[label=(\roman*)]
\item $\emptyset \in \mathcal{I}$;
\item if $I \in \mathcal{I}$ and $I' \subseteq I$, then $I' \in \mathcal{I}$;
\item if $I_1, I_2 \in \mathcal{I}$ with $|I_1| < |I_2|$, then there exists $e \in I_2 \setminus I_1$ such that $I_1 \cup \{e\} \in \mathcal{I}$.
\end{enumerate}
\end{definition}

There are multiple equivalent ways to define a matroid, including descriptions in terms of circuits, the rank function, or bases. We recall these notions below and refer to \cite{Oxley} for a comprehensive discussion on these equivalent definitions.

\begin{definition}
Let $M$ be a matroid on the ground set $[d]$.
\vspace*{-.5em}

\hspace*{-1cm}\parbox{\linewidth}{
\begin{itemize}[label=$\blacktriangleright$]
\item A subset of $[d]$ that is not independent is called \emph{dependent}. The set of all dependent sets of $M$ is denoted by $\mathcal{D}(M)$.\vs
\item A \emph{circuit} is a minimally dependent subset of $[d]$. The set of all circuits is denoted by $\mathcal{C}(M)$.\vs
\item The \emph{rank} of a subset $F \subseteq [d]$, denoted $\rank(F)$, is the size of the largest independent set contained in $F$. The rank of $M$, denoted $\rank(M)$, is defined as $\rank([d])$.\vs 
\item A \emph{basis} is a maximal independent subset of $[d]$. The collection of all bases is denoted by $\mathcal{B}(M)$.\vs
%\item Given $F \subset [d]$, the \emph{rank} of $F$, written $\rank_M(F)$ or simply $\rank(F)$, is the size of the largest independent subset contained in $F$. The \emph{rank of the matroid} $M$, denoted $\rank(M)$, is defined as $\rank([d])$, which equals the size of any basis of $M$.
\item An element $x \in [d]$ belongs to the \emph{closure} of a set $F \subset [d]$, denoted $x \in \closure{F}$, if $\rank(F \cup \{x\}) = \rank(F)$. A set $F$ is called a \emph{flat} if $F = \closure{F}$.\vs
\item The \emph{restriction} of $M$ to a subset $S \subset [d]$ is the matroid on $S$ whose independent sets are those of $M$ contained in $S$. The \emph{deletion} of $S$ from $M$, denoted $M \backslash S$, is defined as the restriction $M|([d] \setminus S)$.\vs
\item Let $x \in [d]$. If $\rank(\{x\}) = 0$, we call $x$ a \emph{loop}. Conversely, $x$ is a \emph{coloop} if it does not belong to any circuit of $M$.  
A subset $\{x,y\} \subset [d]$ is called a \emph{double point} if $\{x,y\}$ is a circuit.
\end{itemize}
}
\end{definition}

\begin{definition}\label{uniform 3}\normalfont
The \emph{uniform matroid} $U_{n,d}$ of rank $n$ on the ground set $[d]$ is defined by: a subset $S \subset [d]$ is independent if and only if $|S| \leq n$. See the matroid $U_{2,7}$ in Figure~\ref{figure uniform and fano} (Left).
\end{definition}

\begin{figure}[H]
    \centering
    \includegraphics[width=0.5\textwidth, trim=0 0 0 0, clip]{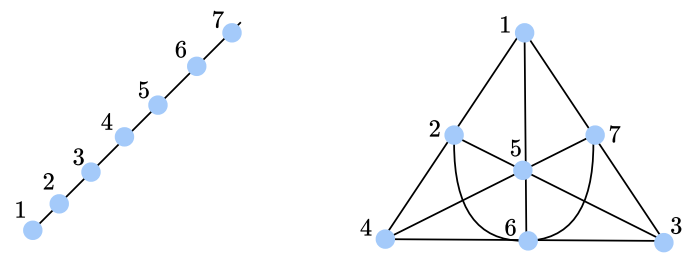}
    \caption{(Left) Uniform matroid $U_{2,7}$; (Right) Fano plane.}
    \label{figure uniform and fano}
\end{figure}

%\begin{definition}\label{dependency}
%Let $N_{1}$ and $N_{2}$ be matroids on $[d]$. We say that $N_{1}\leq N_{2}$ if $\mathcal{D}(N_{1})\subset \mathcal{D}(N_{2})$. This partial order is referred to as the {\em dependency order} on matroids.  
%\end{definition}

\begin{comment}
We now recall the matroid operation of restriction.

\begin{definition}%[Submatroid]
\normalfont \label{subm}
%Let $M$ be a matroid on the ground set $[d]$.
For any subset $S \subseteq [d]$, the {\em restriction} of $M$ to $S$ is the matroid %$\restr{M}{S}$ on $S$ 
 on $S$ whose rank function is the restriction of the rank function of $M$ to $S$. %This matroid is referred to as the restriction of $M$ to $S$. 
Unless stated otherwise, subsets of $[d]$ are assumed to carry this structure and are referred to as \textit{submatroids} of $M$. This submatroid is denoted by $M|S$, or simply $S$ when the context is clear. The {\em deletion} of $S$, denoted $M\setminus S$, is the matroid given by
$M|([d]\setminus S)$.
\end{definition}
\end{comment}

\begin{definition}\label{dependency}
Let $N_{1}$ and $N_{2}$ be matroids on $[d]$. We write $N_{2} \geq N_{1}$ if $\mathcal{D}(N_{2}) \supseteq \mathcal{D}(N_{1})$. This defines a partial order on matroids, known as the \emph{dependency order}.
\end{definition}

\subsection{Paving matroids}

A notable subfamily of matroids is the family of paving matroids. It is conjectured in \cite{mayhew2011asymptotic} that asymptotically almost all matroids are paving; see also \cite{lowrance2013properties, mayhew2011asymptotic, oxley1991ternary, brualdi1972foundations, Oxley, welsh2010matroid}. We refer to Definition~\ref{pav} for their precise definition and for the notion of dependent hyperplanes. The following lemma from \cite{hartmanis1959lattice} provides a hypergraph-theoretic characterization of paving matroids.

\begin{lemma}\label{sub h}
Let $d$ and $n$ be integers with $d \geq n + 1$. Suppose $\mathcal{L}$ is a collection of subsets of $[d]$ such that each $l \in \mathcal{L}$ satisfies $\size{l} \geq n$. Then $\mathcal{L}$ defines the collection of dependent hyperplanes of an $n$-paving matroid $M$ on the ground set $[d]$ if and only if
\begin{equation}\label{condition for paving}
\size{l_1 \cap l_2} \leq n - 2 \quad \text{for all distinct } l_1, l_2 \in \mathcal{L}.
\end{equation}
\end{lemma}

\begin{example}
Consider the collection of subsets  
\[
\mathcal{L}=\{\{1,2,4\},\{1,3,7\},\{1,5,6\},\{2,3,5\},\{2,6,7\},\{3,4,6\},\{4,5,7\}\}.
\]
For $n=3$, it is straightforward to verify that $\mathcal{L}$ satisfies the condition in~\eqref{condition for paving}. Hence, $\mathcal{L}$ defines the collection of dependent hyperplanes of a rank-three paving matroid. This matroid is known as the {\em Fano plane}, depicted in Figure~\ref{figure uniform and fano} (Right).
\end{example}

\begin{definition}\normalfont \label{pav 2}
Let $M$ be a paving matroid of rank $n$ on the ground set $[d]$. We adopt the following terminology:
\begin{itemize}
\item We refer to $M$ as an \emph{$n$-paving matroid}.
\item The set of dependent hyperplanes of $M$ is denoted by $\mathcal{L}_{M}$, or simply $\mathcal{L}$ when the context permits.
\item When $n=3$, the elements of $\mathcal{L}_{M}$ are referred to as \emph{lines}.
\item For each element $p \in [d]$, we write $\mathcal{L}_{p}$ for the collection of dependent hyperplanes containing $p$.
\item The \emph{degree} of an element $p \in [d]$, denoted $\deg(p)$, is defined as $\size{\mathcal{L}_{p}}$.
\item Given a tuple of vectors $\gamma =(\gamma_{1}, \ldots, \gamma_{d})$ in $\CC^{n}$ indexed by $[d]$, we denote by $\gamma_{p}$ the vector corresponding to $p$, and for each $l \in \mathcal{L}_{M}$, we write $\gamma_{l} := \text{span}(\gamma_{p} : p \in l) \subset \CC^{n}$.
\end{itemize}
\end{definition}

%\textcolor{red}{The {\em degree} of $p$, written as deg$(p)$, is defined as $\size{\mathcal{L}_{p}}$. 
%\item 
%For a collection of vectors $\gamma$ in $\CC^{n}$ indexed by $\mathcal{P}_{M}$,  we denote by $\gamma_{p}$ the vector of $\gamma$ corresponding to $p$. For any subspace $l\in \mathcal{L}_{M}$, $\gamma_{l}$ denotes the subspace $\is{\gamma_{p}:p\in l}\subset \CC^{n}$. 
%\end{itemize}
%For clarity, we may omit the index $M$ and 
%simply write $\mathcal{L}$ and $\mathcal{P}$ when the context is clear. }

%\begin{definition}
%{\em Point-line configurations} are defined as paving matroids of rank three. For these matroids, the dependent hyperplanes are referred to as {\em lines}.
%\end{definition}

\begin{example}\label{example lines}
Consider the configurations depicted in Figure~\ref{fig:combined}.
\begin{itemize}
\item The configuration on the left represents a rank-three paving matroid on the ground set $\{1,\ldots,7\}$, with lines  
\[
\mathcal{L} = \{\{1,2,7\}, \{3,4,7\}, \{5,6,7\}\}.
\]

\item The configuration on the right corresponds to a rank-three paving matroid on the ground set $\{1,\ldots,6\}$, with lines  
\[
\mathcal{L}=\{\{1,2,3\},\{1,5,6\},\{3,4,5\},\{2,4,6\}\}.
\]
This configuration is known as the {\em quadrilateral set}, and denoted by $\text{QS}$. All the elements of $\text{QS}$ have degree two.
\end{itemize}
\end{example}

\begin{figure}[h]
    \centering
    \begin{subfigure}[b]{0.3\textwidth}
        \centering
        \begin{tikzpicture}[x=0.75pt,y=0.75pt,yscale=-1,xscale=1]

\tikzset{every picture/.style={line width=0.75pt}} %set default line width to 0.75pt        

%Straight Lines [id:da06101063982657062] 
\draw    (81.69,116.61) -- (191.16,174.3) ;
%Straight Lines [id:da7678393526214486] 
\draw    (77,131.88) -- (224,131.88) ;
%Straight Lines [id:da3419905968139849] 
\draw    (80.13,150.55) -- (191.16,79.27) ;
%Shape: Ellipse [id:dp43020942362366354] 
\draw [fill={rgb, 255:red, 173; green, 216; blue, 230}, fill opacity=1]
(107.34,131.2) .. controls (107.34,133.07) and (108.63,134.58) .. (110.23,134.58) .. controls (111.83,134.58) and (113.12,133.07) .. (113.12,131.2) .. controls (113.12,129.33) and (111.83,127.81) .. (110.23,127.81) .. controls (108.63,127.81) and (107.34,129.33) .. (107.34,131.2) -- cycle ;
%Shape: Ellipse [id:dp7909486902309077] 
\draw [fill={rgb, 255:red, 173; green, 216; blue, 230}, fill opacity=1]
(142.62,149.93) .. controls (142.62,151.8) and (143.91,153.32) .. (145.51,153.32) .. controls (147.11,153.32) and (148.4,151.8) .. (148.4,149.93) .. controls (148.4,148.06) and (147.11,146.55) .. (145.51,146.55) .. controls (143.91,146.55) and (142.62,148.06) .. (142.62,149.93) -- cycle ;
%Shape: Ellipse [id:dp6735689153424713] 
\draw [fill={rgb, 255:red, 173; green, 216; blue, 230}, fill opacity=1]
(173.7,166.79) .. controls (173.7,168.66) and (174.99,170.18) .. (176.59,170.18) .. controls (178.19,170.18) and (179.48,168.66) .. (179.48,166.79) .. controls (179.48,164.92) and (178.19,163.41) .. (176.59,163.41) .. controls (174.99,163.41) and (173.7,164.92) .. (173.7,166.79) -- cycle ;
%Shape: Ellipse [id:dp9808673721534766] 
\draw [fill={rgb, 255:red, 173; green, 216; blue, 230}, fill opacity=1] (201.42,131.2) .. controls (201.42,133.07) and (202.71,134.58) .. (204.31,134.58) .. controls (205.91,134.58) and (207.2,133.07) .. (207.2,131.2) .. controls (207.2,129.33) and (205.91,127.81) .. (204.31,127.81) .. controls (202.71,127.81) and (201.42,129.33) .. (201.42,131.2) -- cycle ;
%Shape: Ellipse [id:dp8560955724937005] 
\draw [fill={rgb, 255:red, 173; green, 216; blue, 230}, fill opacity=1] (167.82,131.2) .. controls (167.82,133.07) and (169.11,134.58) .. (170.71,134.58) .. controls (172.31,134.58) and (173.6,133.07) .. (173.6,131.2) .. controls (173.6,129.33) and (172.31,127.81) .. (170.71,127.81) .. controls (169.11,127.81) and (167.82,129.33) .. (167.82,131.2) -- cycle ;
%Shape: Ellipse [id:dp8522323991300782] 
\draw [fill={rgb, 255:red, 173; green, 216; blue, 230}, fill opacity=1] (177.06,87.18) .. controls (177.06,89.05) and (178.35,90.56) .. (179.95,90.56) .. controls (181.55,90.56) and (182.84,89.05) .. (182.84,87.18) .. controls (182.84,85.31) and (181.55,83.79) .. (179.95,83.79) .. controls (178.35,83.79) and (177.06,85.31) .. (177.06,87.18) -- cycle ;
%Shape: Ellipse [id:dp38553240573402947] 
\draw [fill={rgb, 255:red, 173; green, 216; blue, 230}, fill opacity=1] (145.14,105.91) .. controls (145.14,107.78) and (146.43,109.3) .. (148.03,109.3) .. controls (149.63,109.3) and (150.92,107.78) .. (150.92,105.91) .. controls (150.92,104.04) and (149.63,102.52) .. (148.03,102.52) .. controls (146.43,102.52) and (145.14,104.04) .. (145.14,105.91) -- cycle ;

% Text Node
\draw (166.25,71.57) node [anchor=north west][inner sep=0.75pt]   [align=left] {{\scriptsize $1$}};
% Text Node
\draw (206.28,115.51) node [anchor=north west][inner sep=0.75pt]   [align=left] {{\scriptsize $3$}};
% Text Node
\draw (138.91,153.35) node [anchor=north west][inner sep=0.75pt]   [align=left] {{\scriptsize $6$}};
% Text Node
\draw (170.08,173.62) node [anchor=north west][inner sep=0.75pt]   [align=left] {{\scriptsize $5$}};
% Text Node
\draw (173.44,116.45) node [anchor=north west][inner sep=0.75pt]   [align=left] {{\scriptsize $4$}};
% Text Node
\draw (106.8,111.36) node [anchor=north west][inner sep=0.75pt]   [align=left] {{\scriptsize $7$}};
% Text Node
\draw (136.66,90.59) node [anchor=north west][inner sep=0.75pt]   [align=left] {{\scriptsize $2$}};
% Text Node
%\draw (96.73,190.03) node [anchor=north west][inner sep=0.75pt]   [align=left] {{\footnotesize Pencil of three lines}};

   \end{tikzpicture}
        %\caption{Quadrilateral set}
        \label{fig:quadrilateral 2}
    \end{subfigure}
    %\hfill
    \begin{subfigure}[b]{0.3\textwidth}
        \centering

\tikzset{every picture/.style={line width=0.75pt}} %set default line width to 0.75pt        

\begin{tikzpicture}[x=0.75pt,y=0.75pt,yscale=-1,xscale=1]

%Straight Lines [id:da9613256122302601] 
\draw [line width=0.75]    (247.65,123.92+60) -- (201.96,207.53+60) ;
%Straight Lines [id:da3259117734461473] 
\draw [line width=0.75]    (247.65,123.92+60) -- (291.68,207.53+60) ;
%Straight Lines [id:da2624677482334885] 
\draw [line width=0.75]    (219.23,174.98+60) -- (291.68,207.53+60) ;
%Straight Lines [id:da26466266747028633] 
\draw [line width=0.75]    (274.41,174.98+60) -- (201.96,207.53+60) ;
%Shape: Ellipse [id:dp10345993152561672] 
\draw [fill={rgb, 255:red, 173; green, 216; blue, 230}, fill opacity=1] (244.87,123.92+60) .. controls (244.87,125.68+60) and (246.12,127.11+60) .. (247.65,127.11+60) .. controls (249.19,127.11+60) and (250.44,125.68+60) .. (250.44,123.92+60) .. controls (250.44,122.15+60) and (249.19,120.73+60) .. (247.65,120.73+60) .. controls (246.12,120.73+60) and (244.87,122.15+60) .. (244.87,123.92+60) -- cycle ;
%Shape: Ellipse [id:dp21945184574290333] 
\draw [fill={rgb, 255:red, 173; green, 216; blue, 230}, fill opacity=1] (244.31,187.1+60) .. controls (244.31,188.87+60) and (245.56,190.29+60) .. (247.1,190.29+60) .. controls (248.64,190.29+60) and (249.88,188.87+60) .. (249.88,187.1+60) .. controls (249.88,185.34+60) and (248.64,183.91+60) .. (247.1,183.91+60) .. controls (245.56,183.91+60) and (244.31,185.34+60) .. (244.31,187.1+60) -- cycle ;
%Shape: Ellipse [id:dp033802625652762264] 
\draw [fill={rgb, 255:red, 173; green, 216; blue, 230}, fill opacity=1] (216.45,174.98+60) .. controls (216.45,176.74+60) and (217.69,178.17+60) .. (219.23,178.17+60) .. controls (220.77,178.17+60) and (222.02,176.74+60) .. (222.02,174.98+60) .. controls (222.02,173.21+60) and (220.77,171.79+60) .. (219.23,171.79+60) .. controls (217.69,171.79+60) and (216.45,173.21+60) .. (216.45,174.98+60) -- cycle ;
%Shape: Ellipse [id:dp395767307047196] 
\draw [fill={rgb, 255:red, 173; green, 216; blue, 230}, fill opacity=1] (271.62,174.98+60) .. controls (271.62,176.74+60) and (272.87,178.17+60) .. (274.41,178.17+60) .. controls (275.94,178.17+60) and (277.19,176.74+60) .. (277.19,174.98+60) .. controls (277.19,173.21+60) and (275.94,171.79+60) .. (274.41,171.79+60) .. controls (272.87,171.79+60) and (271.62,173.21+60) .. (271.62,174.98+60) -- cycle ;
%Shape: Ellipse [id:dp5069700831588411] 
\draw [fill={rgb, 255:red, 173; green, 216; blue, 230}, fill opacity=1] (288.9,207.53+60) .. controls (288.9,209.29+60) and (290.14,210.72+60) .. (291.68,210.72+60) .. controls (293.22,210.72+60) and (294.47,209.29+60) .. (294.47,207.53+60) .. controls (294.47,205.77+60) and (293.22,204.34+60) .. (291.68,204.34+60) .. controls (290.14,204.34+60) and (288.9,205.77+60) .. (288.9,207.53+60) -- cycle ;
%Shape: Ellipse [id:dp3796249357557302] 
\draw [fill={rgb, 255:red, 173; green, 216; blue, 230}, fill opacity=1] (199.17,207.53+60) .. controls (199.17,209.29+60) and (200.42,210.72+60) .. (201.96,210.72+60) .. controls (203.49,210.72+60) and (204.74,209.29+60) .. (204.74,207.53+60) .. controls (204.74,205.77+60) and (203.49,204.34+60) .. (201.96,204.34+60) .. controls (200.42,204.34+60) and (199.17,205.77+60) .. (199.17,207.53+60) -- cycle ;

% Text Node
\draw (243.29,107.07+60) node [anchor=north west][inner sep=0.75pt]   [align=left] {{\scriptsize $1$}};
% Text Node
\draw (200.77,165.67+60) node [anchor=north west][inner sep=0.75pt]   [align=left] {{\scriptsize $2$}};
% Text Node
\draw (184.94,205.77+60) node [anchor=north west][inner sep=0.75pt]   [align=left] {{\scriptsize $3$}};
% Text Node
\draw (241.88,169.35+60) node [anchor=north west][inner sep=0.75pt]   [align=left] {{\scriptsize $4$}};
% Text Node
\draw (277.98,164.77+60) node [anchor=north west][inner sep=0.75pt]   [align=left] {{\scriptsize $5$}};
% Text Node
\draw (292.84,206.21+60) node [anchor=north west][inner sep=0.75pt]   [align=left] {{\scriptsize $6$}};
% Text Node
%\draw (203.1,213.83) node [anchor=north west][inner sep=0.75pt]   [align=left] {{\footnotesize Quadrilateral set}};

\end{tikzpicture}
       % \caption{Pascal configuration}
        \label{fig:pascal}
    \end{subfigure}
\caption{(Left) Three concurrent lines; (Right) Quadrilateral set.}
    \label{fig:combined}
\end{figure}
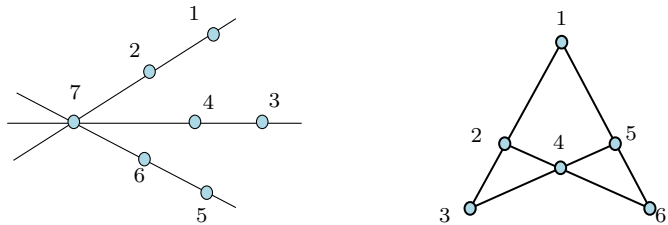

Given a paving matroid $M$ and a subset $L \subset \mathcal{L}_M$ of its dependent hyperplanes, one can naturally associate to them another paving matroid supported on the union of the elements of $L$.

\begin{definition}\normalfont\label{submatroid of hyperplanes}
Let $M$ be an $n$-paving matroid on $[d]$, and let $L \subset \mathcal{L}_M$ be a collection containing at least two dependent hyperplanes. 
By Lemma~\ref{sub h}, the set $\cup_{l \in L} l \subset [d]$ inherits the structure of an $n$-paving matroid with $L$ as its set of dependent hyperplanes. We denote this matroid by $M^L$.
\end{definition}

\begin{example}
Consider the quadrilateral set $\textup{QS}$ from Example~\ref{example lines}, and consider the subset  
\[
L=\{\{1,2,3\},\{1,5,6\}\}
\]
of its dependent hyperplanes. The matroid $\text{QS}^{L}$ is then the rank-three paving matroid on the ground set $\{1,2,3,5,6\}$, whose set of lines is precisely $L$.
\end{example}

%It is important to note that this matroid structure on $\mathcal{P}_{M^{L}}=\cup_{l\in L}l$ is distinct from the one given as a submatroid of $M$, in Definition~\ref{subm}.

%We now introduce a subfamily of paving matroids that plays a key role in this note.

%\begin{definition}
%A \emph{tame paving matroid} is a paving matroid where no point has degree greater than two, or equivalently, the intersection of any three dependent hyperplanes is empty.
%\end{definition}

We conclude this subsection by recalling the notion of tame paving matroids from Definition~\ref{2 paving}, as it will play a central role in what follows.

\subsection{Nilpotent and liftable paving matroids}\label{subsection nilpotent and liftable}

In this subsection, we review the notions of \emph{nilpotent} and \emph{liftable} paving matroids, introduced in \cite{LIWSKI2026102484} and \cite{Fatemeh4}, which will be important in what follows.

\begin{definition}\textup{\cite[Definition~3.2]{LIWSKI2026102484}}\label{def nilpotent}
Let $M$ be a paving matroid on $[d]$. We define 
    \[S_M=\{p \in [d] \mid \size{\mathcal{L}_{p}} > 1\}.\]
The \textit{nilpotent chain} of $M$ is defined as the following sequence of submatroids of $M$: 
$$M_{0}=M,\quad M_{1}=M|S_{M},\quad \text{and }\ M_{j+1}=M|S_{M_{j}} \ \text{ for every $j\geq 1$.}$$
We say that $M$ is \textit{nilpotent} if $M_{j}=\emptyset$ for some $j$.
\end{definition}

For the following definition, we recall the variety $V_{\mathcal{C}(M)}$ introduced in Definition~\ref{cir}.

\begin{definition}\cite{Fatemeh4}\label{lifta}
Let $M$ be an $n$-paving matroid on $[d]$. We introduce the following notions:
\vspace*{-.5em}

\hspace*{-1cm}\parbox{\linewidth}{
\begin{itemize}[label=$\blacktriangleright$]
\item Let $\gamma = (\gamma_1, \ldots, \gamma_d)$ be a tuple of vectors in $\CC^n$, and let $q \in \CC^n$. A tuple $\widetilde{\gamma} = (\widetilde{\gamma}_1, \ldots, \widetilde{\gamma}_d)$ is called a \emph{lifting} of $\gamma$ from $q$ if, for each $p \in [d]$, there exists $z_p \in \CC$ such that $\widetilde{\gamma}_p = \gamma_p + z_p q$. The lifting is said to be \emph{non-degenerate} if the lifted vectors do not all lie in a common hyperplane.\vs

\item The matroid $M$ is \emph{liftable} if, for every tuple $\gamma = (\gamma_1, \ldots, \gamma_d)$ spanning a hyperplane $H \subset \CC^n$, and for every $q \notin H$, there exists a non-degenerate lifting 
\[
\widetilde{\gamma} = (\gamma_1 + z_1 q, \ldots, \gamma_d + z_d q) \in V_{\mathcal{C}(M)}.
\]
\end{itemize}}
\end{definition}

The following result establishes a connection between liftable and nilpotent paving matroids.

\begin{proposition}\textup{\cite[Proposition~3.16]{LIWSKI2026102484}}\label{new propo}
Every nilpotent paving matroid is liftable.
\end{proposition}

\begin{example}
Consider the rank-three matroids depicted in Figure~\ref{fig:combined}.
\begin{itemize}
\item For the matroid $M$ on the left, we have $S_{M}=\{7\}$ and $S_{S_{M}}=\emptyset$. Hence, $M$ is nilpotent, and therefore liftable as well.
\item For the quadrilateral set $\textup{QS}$ on the right, every element has degree two, so it is not nilpotent. Furthermore, this matroid is not liftable, see \textup{\cite[Example~3.6]{Fatemeh3}}.
\end{itemize}
\end{example}

\section{Quasi-paving matroids}\label{sec 3}

In this section, we introduce the class of \emph{quasi-paving matroids} and examine some of their key properties. These matroids admit a hypergraph-theoretic construction analogous to that of paving matroids described in Lemma~\ref{sub h}. The class of quasi-paving matroids extends the notion of tame paving matroids from Definition~\ref{2 paving}, and will play a central role in addressing Question~\ref{question stn}.

\begin{theorem}\label{quasi}
Let $\mathcal{H} = \{H_1, \ldots, H_k\}$ be a collection of subsets of $[d]$ such that the intersection of any three of them is empty, and let $n \leq d$ be a fixed positive integer. Then the following collection of subsets of $[d]$ forms the set of circuits of a matroid:
\begin{itemize}
\item[{\bf (1)}] Any $(n{-}1)$-element subset of $[d]$ contained in the intersection of two elements of $\mathcal{H}$.
\item[{\bf (2)}] Any $n$-element subset of some $H_i \in \mathcal{H}$ that does not contain a subset of Type~1.
\item[{\bf (3)}] Any $(n{+}1)$-element subset of $[d]$ that does not contain a subset of Type~1 or Type~2.
\end{itemize}
\end{theorem}

\begin{proof}
Let $\mathcal{C}$ denote the collection of subsets of $[d]$ described in Types~1--3 above. To show that $\mathcal{C}$ is the set of circuits of a matroid, we must verify the circuit exchange axiom: for any distinct $C_1, C_2 \in \mathcal{C}$ and any $x \in C_1 \cap C_2$, there exists $C_3 \in \mathcal{C}$ such that $C_3 \subset (C_1 \cup C_2) \setminus \{x\}$.

\medskip
\noindent
{\bf Case 1.} At least one of $C_1$, $C_2$ is of Type~3. Then $|(C_1 \cup C_2) \setminus \{x\}| \geq n+1$, so a subset of $\mathcal{C}$ lies within this set.

\medskip
\noindent
{\bf Case 2.} Both $C_1$ and $C_2$ are of Type~1. Suppose $C_1 \subset H_i \cap H_j$ and $C_2 \subset H_r \cap H_l$ for some $i, j, r, l \in [k]$. Since $x \in C_1 \cap C_2 \subset H_i \cap H_j \cap H_r \cap H_l$ and the intersection of any three elements of $\mathcal{H}$ is empty, we must have $\{i, j\} = \{r, l\}$. Hence $C_1, C_2 \subset H_i \cap H_j$. Since $\size{C_{1}}=\size{C_{2}}=n-1$ it follows that $\size{(C_{1}\cup C_{2})\backslash \{x\}}\geq n-1$, which implies that $(C_{1}\cup C_{2})\backslash \{x\}$ contains a subset of Type~1.

\medskip
\noindent
{\bf Case 3.} One set is of Type~2, the other of Type~1. Suppose $C_1 \subset H_i$ is of Type~2 and $C_2 \subset H_j \cap H_l$ is of Type~1. Since $x \in C_1 \cap C_2 \subset H_i \cap H_j \cap H_l$ and $x$ lies in at most two elements of $\mathcal{H}$, we must have $i \in \{j, l\}$; assume without loss of generality $i = l$. Since $\size{C_{1}}=n$,  
 we have $\size{(C_{1}\cup C_{2})\backslash \{x\}}\geq n$, and since $(C_{1}\cup C_{2})\backslash \{x\}\subset H_{i}$, it follows that $(C_{1}\cup C_{2})\backslash \{x\}$ contains a subset of Type~1 or Type~2.

\medskip
\noindent
{\bf Case 4.} Both $C_1$ and $C_2$ are of Type~2, with $C_1 \subset H_i$ and $C_2 \subset H_j$. Since $|C_1| = |C_2| = n$, we have $|(C_1 \cup C_2) \setminus \{x\}| \geq n$.

\smallskip
\noindent
{\bf Case 4.1.} If $|(C_1 \cup C_2) \setminus \{x\}| \geq n+1$, then an element of $\mathcal{C}$ is contained within $(C_1 \cup C_2) \setminus \{x\}$.

\smallskip
\noindent
{\bf Case 4.2.} Suppose $|(C_1 \cup C_2) \setminus \{x\}| = n$, implying $|C_1 \cap C_2| = n - 1$. We distinguish two subcases:
\begin{itemize}
\item If $i = j$, then $(C_1 \cup C_2) \setminus \{x\} \subset H_i$, and the resulting set is either of Type~2 or contains a subset of Type~1.
\item If $i \ne j$, then $C_1 \cap C_2 \subset H_i \cap H_j$ has size $n-1$, so it is a Type~1 set. Therefore, $C_{1}$ includes a set of Type~1, but this contradicts the assumption that $C_{1}$ is of Type~2. 
\end{itemize}

\noindent This completes the verification of the circuit exchange property and hence the proof.
\end{proof}

We now introduce the class of \emph{quasi-paving matroids}.

\begin{definition}
A matroid is called a \emph{quasi-paving matroid} if its set of circuits is given by the construction in Theorem~\ref{quasi}. More precisely, we refer to such a matroid as an \emph{$n$-quasi-paving matroid} when the construction involves the integer~$n$. Note that an $n$-quasi-paving matroid does not necessarily have rank~$n$.
\end{definition}

If a quasi-paving matroid $M$ arises from a collection of subsets $\mathcal{H}$, we refer to $\mathcal{H}$ as a \emph{representation} of $M$.

\begin{figure}[H]
    \centering
    \includegraphics[width=0.5\textwidth, trim=0 0 0 0, clip]{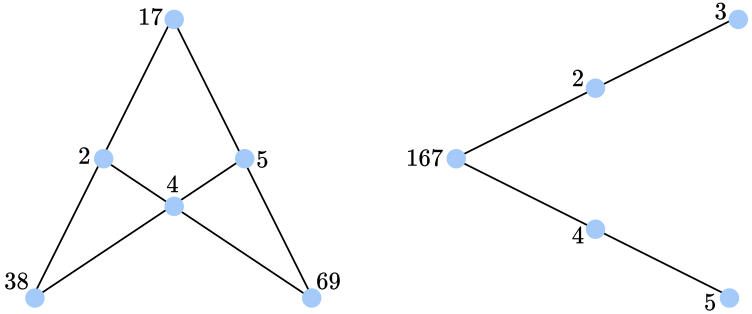}
    \caption{Two $3$-quasi-paving matroids}
    \label{new figure 5}
\end{figure}

\begin{example}\label{example 3-quasi}
Consider the hypergraphs \[\mathcal{H}_{1} = \{\{1,2,3,7,8\}, \{1,5,6,7,9\}, \{2,4,6,9\}, \{3,4,5,8\}\} \quad \text{and} \quad \mathcal{H}_{2} = \{\{1,4,5,6,7\}, \{1,2,3,6,7\}\}\] on the ground sets $[9]$ and $[7]$, respectively. Applying Theorem~\ref{quasi} with $n = 3$ to each hypergraph yields the quasi-paving matroids depicted in Figure~\ref{new figure 5}. In these figures, coinciding points represent circuits of size two, while three elements aligned on the same line correspond to dependencies of size three.
\end{example}

The class of quasi-paving matroids properly contains the class of tame paving matroids.

\begin{proposition}
Let $M$ be a tame paving matroid. Then, $M$ is quasi-paving.
\end{proposition}

\begin{proof}
Let $n$ denote the rank of $M$. Consider the hypergraph $\mathcal{L}_M$ formed by the collection of dependent hyperplanes of $M$. By the characterization of paving matroids in Lemma~\ref{sub h}, for any distinct $l_1, l_2 \in \mathcal{L}_M$ we have $\lvert l_1 \cap l_2 \rvert \leq n - 2$. Moreover, since $M$ is tame, the intersection of any three distinct elements of $\mathcal{L}_M$ is empty. 

It follows directly that $M$ coincides with the quasi-paving matroid constructed by applying Theorem~\ref{quasi} to the hypergraph $\mathcal{L}_M$ and the integer $n$.
\end{proof}

The following proposition establishes fundamental properties of quasi-paving matroids.

\begin{proposition}\label{is a flat} 
Let $\mathcal{H} = \{H_1, \ldots, H_k\}$ be a representation of an $n$-quasi-paving matroid $M$ over the ground set $[d]$. Then, the following statements hold:
\begin{itemize}
\item[{\rm (i)}] For any subset $Z \subseteq [d]$, the deletion $M \setminus Z$ is also quasi-paving.
\item[{\rm (ii)}] For any distinct $i,j \in [k]$ with $\lvert H_i \cap H_j \rvert \geq n-2$, the intersection $H_i \cap H_j$ is a flat of $M$ isomorphic to the uniform matroid $U_{n-2,\, \lvert H_i \cap H_j \rvert}$.
\item[{\rm (iii)}] If $\lvert H_i \rvert \geq n-1$ and $H_i \not\subseteq H_j$ for all $j \neq i$, then $\operatorname{rank}(H_i) = n-1$.
\end{itemize}
\end{proposition}

\begin{proof}
(i) The deletion $M \setminus Z$ is obtained as the $n$-quasi-paving matroid on the ground set $[d] \setminus Z$ induced by the hypergraph 
\[
\mathcal{H}' = \{ H_1 \setminus Z, \ldots, H_k \setminus Z \}.
\]

\medskip
(ii) From the hypergraph representation of $M$, any subset of $n-1$ elements contained in $H_i \cap H_j$ forms a circuit. Furthermore, by Theorem~\ref{quasi}, every subset of $[d]$ with size at most $n-2$ is independent. Consequently, the restriction of $M$ to $H_i \cap H_j$ is isomorphic to the uniform matroid $U_{n-2,\, \lvert H_i \cap H_j \rvert}$.

It remains to show that $H_i \cap H_j$ is a flat of $M$. To this end, suppose there exists $x \in [d]$ such that
\[
\operatorname{rank} \big( (H_i \cap H_j) \cup \{x\} \big) = \operatorname{rank}(H_i \cap H_j) = n-2.
\]
Then every subset of $n-1$ elements of $(H_i \cap H_j) \cup \{x\}$ forms a Type~1 circuit, implying these elements lie in the intersection of exactly two elements of $\mathcal{H}$. Since each element of $[d]$ belongs to at most two elements of $\mathcal{H}$, these must be $H_i$ and $H_j$, forcing $x \in H_i \cap H_j$. Thus, $H_i \cap H_j$ is a flat.

\medskip
(iii) We consider two cases:

\medskip
\textbf{Case 1.} Suppose there exists $j \neq i$ such that $\lvert H_i \cap H_j \rvert \geq n-1$. Let $B_1 \subseteq H_i \cap H_j$ be of size $n-2$, and pick $x \in H_i \setminus H_j$. Then $B = B_1 \cup \{x\} \subseteq H_i$ and $B \not\subseteq H_i \cap H_j$. Since $B$ is not contained in the intersection of two elements of $\mathcal{H}$, $B$ is independent, which implies $\operatorname{rank}(H_i) \geq n-1$. Since $\operatorname{rank}(H_i) \leq n-1$ by definition, this yields $\operatorname{rank}(H_i) = n-1$.

\medskip
\textbf{Case 2.} Suppose $\lvert H_i \cap H_j \rvert \leq n-2$ for every $j \neq i$. Then, any subset of $n-1$ elements of $H_i$ is independent because it cannot be contained in the intersection of two elements of $\mathcal{H}$. Hence, $\operatorname{rank}(H_i) = n-1$.
\end{proof}

The following proposition provides an alternative description of $n$-quasi-paving matroids of rank $n$. To state it, we first recall the notion of a \emph{principal extension}.  

\begin{definition}\label{def principal extension}
Let $M'$ be a matroid on the ground set $E'$, and let $F$ be a flat of $M'$. The \emph{principal extension} of $M'$ by an element $a$ into $F$ is the matroid $M = M' +_F a$ on the ground set $E = E' \cup \{a\}$, whose collection of bases is
\[
\mathcal{B}(M) = \mathcal{B}(M') \;\cup\; \left\{ (\lambda \setminus \{b\}) \cup \{a\} \;\middle|\; \lambda \in \mathcal{B}(M'), \ b \in \lambda \cap F \right\}.
\]
\end{definition}

\begin{proposition}\label{extension}
Every $n$-quasi-paving matroid of rank $n$ arises from a tame paving matroid by successive principal extensions along flats of rank $n-2$.
\end{proposition}

\begin{proof}
We proceed by induction on $d$, the size of the ground set. The base case $d = n$ is immediate.

\medskip
\noindent
\textbf{Inductive step.} Suppose the result holds for all $n$-quasi-paving matroids of rank $n$ on fewer than $d$ elements. Let $M$ be an $n$-quasi-paving matroid of rank $n$ on the ground set $[d]$, with representation $\mathcal{H} = \{H_1,\ldots,H_k\}$. We distinguish two cases.

\medskip
\noindent
\textbf{Case 1.} If $|H_i \cap H_j| \leq n-2$ for all $i,j \in [k]$, then $M$ is already a tame paving matroid and the claim holds trivially.

\medskip
\noindent
\textbf{Case 2.} Suppose there exist $i,j \in [k]$ with $|H_i \cap H_j| \geq n-1$, and choose any $a \in H_i \cap H_j$. Define $M' = M \setminus \{a\}$, which is again an $n$-quasi-paving matroid of rank $n$, with representation
\[
\mathcal{H}' = \{H_1 \setminus \{a\}, \ldots, H_k \setminus \{a\}\}.
\]
Let $F = (H_i \cap H_j) \setminus \{a\}$. By Proposition~\ref{is a flat}, $F$ is a flat of $M'$ of rank $n-2$.

\medskip
\noindent
\textbf{Claim.} We have $M = M' +_F a$.

\medskip
To verify this, we examine the bases of $M$ and $M'$. Let $\lambda \subset [d] \setminus \{a\}$. Then $\lambda \in \mathcal{B}(M')$ if and only if $|\lambda| = n$ and $\lambda$ is neither of Type~2 nor contains a set of Type~1, which is precisely the condition for $\lambda \in \mathcal{B}(M)$ when $a \notin \lambda$.

Now consider $\lambda \in \mathcal{B}(M)$ with $a \in \lambda$. Since $\operatorname{rank}(H_i \cap H_j) = n - 2$ and $|\lambda| = n$, we have $|\lambda \cap (H_i \cap H_j)| \leq n - 2$. Moreover, since $a \in \lambda$ and $|F| \geq n - 2$, this implies that there exists $b \in F \setminus \lambda$. Define $\lambda' = (\lambda \cup \{b\}) \setminus \{a\}$. We claim that $\lambda' \in \mathcal{B}(M')$:

\begin{itemize}
\item Clearly, $|\lambda'| = n$.
\item Suppose that $\lambda'$ is of Type~2. Then, there exists some $l\in [k]$ such that $\lambda' \subset H_\ell$. Then, since $a$ and $b$ appear in the same members of $\mathcal{H}$, it follows that $\lambda \subset H_\ell$, contradicting the assumption that $\lambda$ is independent.
\item Suppose $\lambda'$ contains a circuit $C$ of Type~1. If $b \notin C$, it follows that also $a \notin C$ given that $a$ and $b$ appear in the same members of $\mathcal{H}$.
Then, since $\size{C}=n-1$ and $\size{\lambda}=n$, it follows that $C = \lambda' \setminus \{b\} = \lambda \setminus \{a\}$, implying that $\lambda$ contains a circuit, again a contradiction. If $b \in C$, it follows that $C \subset H_{i} \cap H_{j}$, given that $H_{i}$ and $H_{j}$ are the unique two members of $\mathcal{H}$ containing $b$. Moreover, since $C \subset \lambda'$ and $a \notin \lambda'$, we have $a \notin C$, which implies $C \subset F$. Furthermore, since $|C| = n-1$, we conclude that $\lambda$ contains $n-1$ elements of $F$, contradicting $\operatorname{rank}(F) = n-2$.
\end{itemize}

Thus $\lambda \rq \in \mathcal{B}(M\rq)$. By applying the same argument, the reciprocal can be easily demonstrated. Hence, we conclude that 
\[
\mathcal{B}(M) = \mathcal{B}(M') \cup \left\{ (\lambda \setminus \{b\}) \cup \{a\} : \lambda \in \mathcal{B}(M'),\ b \in \lambda \cap F \right\},
\]
which is precisely the set of bases of $M' +_F a$. Hence, $M = M' +_F a$.

\smallskip
By the inductive hypothesis, $M'$ is obtained from a tame paving matroid by a sequence of principal extensions along flats of rank $n-2$. Therefore, the same holds for $M$, completing the proof.
\end{proof}

\begin{example}\label{example 3-quasi 2}
Consider the $3$-quasi-paving matroids depicted in Figure~\ref{new figure 5}, labeled $M_1$ and $M_2$ from left to right. As discussed in Example~\ref{example 3-quasi}, both matroids arise from the hypergraph construction described in Theorem~\ref{quasi}. Alternatively, following the description given in Proposition~\ref{extension}, they can be obtained via principal extensions as follows:
\begin{itemize}
\item The matroid $M_1$ is obtained from the rank-three matroid $QS$ shown in Figure~\ref{fig:combined} (Right) by performing principal extensions of the elements $7$, $8$, and $9$ over the rank-one flats $\{1\}$, $\{3\}$, and $\{6\}$, respectively. 
\item Similarly, the matroid $M_2$ arises from the rank-three matroid $M_2'$ on the ground set $[5]$, defined by the lines $\{1,2,3\}$ and $\{1,4,5\}$, by performing principal extensions of the elements $6$ and $7$ over the rank-one flat $\{1\}$.
\end{itemize}
\end{example}

Applying Proposition~\ref{extension}, we obtain the following corollary, which ensures both realizability and irreducibility for the matroid varieties of $n$-quasi-paving matroids of rank $n$.

\begin{corollary}\label{coro ir}
Let $M$ be an $n$-quasi-paving matroid of rank $n$. Then $M$ is realizable, and its matroid variety $V_M$ is irreducible. 
\end{corollary}

\begin{proof}
It was shown in 
\textup{\cite[Theorem~4.19]{LIWSKI2026102484}} that tame paving matroids are realizable and have irreducible matroid varieties. Furthermore, \textup{\cite[Lemma~5.5]{clarke2021matroid}} and \textup{\cite[Proposition~3.3]{corey2023singular}} establish that principal extensions preserve both realizability and irreducibility. The claim then follows immediately from Proposition~\ref{extension}.
\end{proof}

\section{Decomposition of circuit varieties}\label{sec 4}

In this section, we describe the irreducible components of the circuit varieties of tame paving matroids, thereby answering Question~\ref{question 2 paving}. This result will later be applied to the variety $V_{\Delta^{s,t}}$ introduced in Definition~\ref{setup}, providing a solution to Question~\ref{question stn}.

Our approach begins by associating to each partition of the collection of dependent hyperplanes of a tame paving matroid a corresponding quasi-paving matroid, as defined below.

\begin{definition}\label{constr}
Let $M$ be a tame paving matroid of rank $n$ on $[d]$, and let $\mathcal{Q} = \{Q_{1}, \ldots, Q_{k}\}$ be a partition of its set of dependent hyperplanes. Define the hypergraph $\mathcal{H} = \{H_{1}, \ldots, H_{k}\}$ on $[d]$ by setting
\[
H_{i} = \bigcup_{l \in Q_{i}} l,
\]
for each $i \in [k]$. Note that every element $j \in [d]$ belongs to at most two members of $\mathcal{H}$. We denote by $M(\mathcal{Q})$ the quasi-paving matroid obtained by applying the construction in Theorem~\ref{quasi} to the hypergraph $\mathcal{H}$ and the integer $n$.
\end{definition}

\begin{figure}[H]
    \centering
    \includegraphics[width=0.5\textwidth, trim=0 0 0 0, clip]{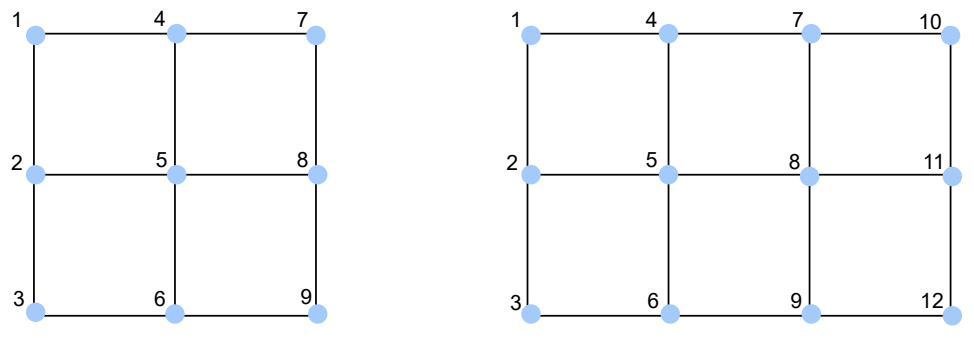}
    \caption{(Left) $3\times 3$ grid; (Right) $3\times 4$ grid}
    \label{3 times 3 and 3 times 4 grids}
\end{figure}

\begin{example}\label{example MQ}
Consider the rank-three matroids depicted in Figure~\ref{3 times 3 and 3 times 4 grids}, denoted by $M$ and $N$, with ground sets $[9]$ and $[12]$, respectively. In both cases, three collinear points represent a circuit of size three. Observe that $M$ and $N$ are tame paving matroids, and that their dependent hyperplanes correspond precisely to the lines of the grids.

\begin{itemize}
\item For $M$, consider the partition
\[
\mathcal{Q}=\big\{\{\{1,4,7\},\{2,5,8\},\{3,6,9\}\},\{\{1,2,3\}\},\{\{4,5,6\}\},\{\{7,8,9\}\}\big\}
\]
of the set of dependent hyperplanes. By Definition~\ref{constr}, the associated quasi-paving matroid $M(\mathcal{Q})$ is a rank-two matroid in which each of the triples $\{1,2,3\}$, $\{4,5,6\}$, and $\{7,8,9\}$ forms a parallel class; that is, within each triple, every pair of elements is dependent.

\item For $N$, consider the partition
\[
\mathcal{Q}=\big\{\{\{1,2,3\},\{4,5,6\}\},\{\{7,8,9\},\{10,11,12\}\},\{\{1,4,7,10\}\},\{\{2,5,8,11\}\},\{\{3,6,9,12\}\}\big\}
\]
of the set of dependent hyperplanes. By Definition~\ref{constr}, the associated quasi-paving matroid $N(\mathcal{Q})$ is a rank-three matroid with double points $\{1,4\}$, $\{2,5\}$, $\{3,6\}$, $\{7,10\}$, $\{8,11\}$, and $\{9,12\}$, and in which the sets $\{1,2,3\}$ and $\{7,8,9\}$ are dependent.
\end{itemize}
\end{example}

As Example~\ref{example MQ} may suggest, the matroid $M(\mathcal{Q})$ lies above $M$ in the dependency order introduced in Definition~\ref{dependency}, as established by the following lemma.

\begin{lemma}\label{mayor igual}
We have $M(\mathcal{Q}) \geq M$.
\end{lemma}

\begin{proof}
We must verify that $\mathcal{D}(M) \subset \mathcal{D}(M(\mathcal{Q}))$. Since $\operatorname{rank}(M(\mathcal{Q})) \leq n$, every dependent set of size at least $n+1$ in $M$ remains dependent in $M(\mathcal{Q})$. Moreover, because $M$ is a paving matroid, its only dependent sets of size at most $n$ are the $n$-element subsets of its dependent hyperplanes. It thus suffices to show that for any $l \in \mathcal{L}$, we have $\operatorname{rank}_{M(\mathcal{Q})}(l) \leq n-1$.

By construction (see Definition~\ref{constr}), there exists some $H_i \in \mathcal{H}$ such that $l \subset H_i$. Therefore,
\[
\operatorname{rank}_{M(\mathcal{Q})}(l) \leq \operatorname{rank}_{M(\mathcal{Q})}(H_i) \leq n-1,
\]
as required.
\end{proof}

We now introduce the notion of \emph{nice} partitions. For this, recall Definitions~\ref{submatroid of hyperplanes} and~\ref{lifta}.

\begin{definition}\label{nice}
Let $\mathcal{Q} = \{Q_{1}, \ldots, Q_{k}\}$ be a partition of the set of dependent hyperplanes of a paving matroid $M$. We say that $\mathcal{Q}$ is \emph{nice} if the following conditions are satisfied:
\begin{itemize}
\item[{\rm (i)}] For each $i \in [k]$ such that $\size{Q_i} \geq 2$, the associated matroid $M^{Q_i}$ is not liftable.
\item[{\rm (ii)}] There does not exist an index $i \in [k]$ and a hyperplane $l \notin Q_i$ such that $l \subseteq \cup_{l' \in Q_i} l'$.
\end{itemize}
\end{definition}

\begin{example}
Let $M$ and $N$ denote the matroids depicted on the left and right of Figure~\ref{3 times 3 and 3 times 4 grids}, respectively. 
\begin{itemize}
\item For $M$, consider the trivial partition 
\begin{equation}\label{trivial partition}
\mathcal{Q}=\{\{\{1,2,3\},\{4,5,6\},\{7,8,9\},\{1,4,7\},\{2,5,8\},\{3,6,9\}\}\}
\end{equation}
of its set of dependent hyperplanes. Equivalently, $\mathcal{Q}$ consists of a single part $Q_{1}$ containing all the dependent hyperplanes of $M$. Clearly, $M^{Q_{1}}=M$. Since $M$ is liftable (see~\cite[Example~3.10]{Fatemeh3}), it follows that $\mathcal{Q}$ is not a nice partition.

\item For $N$, consider the partition
\[
\mathcal{Q}_{1}=\{\{\{1,4,7,10\},\{2,5,8,11\},\{3,6,9,12\}\},\{\{1,2,3\}\},\{\{4,5,6\}\},\{\{7,8,9\}\},\{\{10,11,12\}\}\}
\]
of its set of dependent hyperplanes. The partition $\mathcal{Q}_{1}$ is not nice, since
\[
\{1,2,3\}\subseteq \{1,4,7,10\}\cup \{2,5,8,11\}\cup \{3,6,9,12\},
\]
which violates condition~(ii) of Definition~\ref{nice}. On the other hand, consider the trivial partition
\begin{equation}\label{trivial partition 2}
\mathcal{Q}_{2}=\{\{\{1,4,7,10\},\{2,5,8,11\},\{3,6,9,12\},\{1,2,3\},\{4,5,6\},\{7,8,9\},\{10,11,12\}\}\}.
\end{equation}
Let $Q_{1}$ denote the unique element of $\mathcal{Q}_{2}$, namely the collection of all dependent hyperplanes of $N$. Clearly, $N^{Q_{1}}=N$. Since $N$ is not liftable (see~\cite[Example~3.6]{Fatemeh3}), it follows that $\mathcal{Q}_{2}$ is a nice partition.
\end{itemize}
\end{example}

We are now ready to state the main result of this section, which provides an answer to Question~\ref{question 2 paving}. The technical proof is deferred to Section~\ref{appendix}.

\begin{theorem}\label{Theorem}
Let $M$ be a tame paving matroid. The circuit variety $V_{\mathcal{C}(M)}$ has the following irredundant irreducible decomposition:
\begin{equation}\label{deco teo1}
V_{\mathcal{C}(M)} = \bigcup_{\mathcal{Q}} V_{M(\mathcal{Q})},
\end{equation}
where the union is taken over all nice partitions $\mathcal{Q}$ of the dependent hyperplanes of $M$.
\end{theorem}

Theorem~\ref{Theorem} shows that the irreducible components of $V_{\mathcal{C}(M)}$ are precisely the matroid varieties corresponding to the quasi-paving matroids $M(\mathcal{Q})$ associated with the nice partitions $\mathcal{Q}$ of the dependent hyperplanes of $M$. Although this result answers Question~\ref{question 2 paving}, the description it provides is not fully explicit, since identifying all nice partitions is a nontrivial task and no systematic procedure is currently available. Nevertheless, as we will see in the next section, such an explicit description can indeed be obtained in the context of Question~\ref{question stn} for the case $s=t=n=3$.

\begin{example}\label{example: decomposition grids}
Consider the rank-three matroids depicted in Figure~\ref{3 times 3 and 3 times 4 grids}, denoted by $M$ and $N$. Both are tame paving matroids. Applying Theorem~\ref{Theorem} to these examples, we obtain the following:

\begin{itemize}
\item For $M$, the unique nice partition of its set of dependent hyperplanes is the discrete partition 
\[
\mathcal{Q}=\big\{\{\{1,2,3\}\},\{\{4,5,6\}\},\{\{7,8,9\}\},\{\{1,4,7\}\},\{\{2,5,8\}\},\{\{3,6,9\}\}\big\}.
\]
In this case, $M(\mathcal{Q})=M$, and hence 
\[
V_{\mathcal{C}(M)}=V_{M}.
\]

\item For $N$, there are precisely two nice partitions of its set of dependent hyperplanes:
\begin{itemize}
\item[{\rm (i)}] the trivial partition $\mathcal{Q}_{1}$ from~\eqref{trivial partition 2},  
\item[{\rm (ii)}] the discrete partition
\[
\mathcal{Q}_{2} = \Bigg\{
\begin{aligned}
&\{\{1,4,7,10\}\},\ \{\{2,5,8,11\}\},\ \{\{3,6,9,12\}\},\\
&\{\{1,2,3\}\},\ \{\{4,5,6\}\},\ \{\{7,8,9\}\},\ \{\{10,11,12\}\}
\end{aligned}
\Bigg\}.
\]
\end{itemize}
In this case, $N(\mathcal{Q}_{1})=U_{2,12}$ and $N(\mathcal{Q}_{2})=N$. Therefore, by Equation~\eqref{deco teo1}, the circuit variety of $N$ has the following irredundant irreducible decomposition
\[
V_{\mathcal{C}(N)}=V_{U_{2,12}}\cup V_{N}.
\]
\end{itemize}
\end{example}

We emphasize that Example~\ref{example: decomposition grids} extends one of the main results of \cite{Fatemeh3}, where the irreducible decomposition of the circuit variety of the $3\times 4$ grid matroid is obtained (see \textup{\cite[Corollary~5.8]{Fatemeh3}}). In the next section, we derive the explicit decomposition for grid matroids of arbitrary size.

\section{Decomposition of CI Ideals}\label{sec 5}

In this section, we study the variety $V_{\Delta^{s,t}}$ introduced in Definition~\ref{setup}, with the goal of describing its irreducible components in the case where $s = t = n=3 \leq \min\{k, l\}$. This will be achieved in Theorem~\ref{Theorem 2} by applying Theorem~\ref{Theorem} in this specific context.

Throughout, we fix $s = t = n=3$. We also work under the notation established in Definition~\ref{setup}, and for brevity, we write $\Delta$ in place of $\Delta^{s,t}$. In this setting, the collection $\Delta$ takes the form
\begin{equation}\label{def delta}
\Delta = \bigcup_{i \in [k]} \binom{R_i}{3} \cup \bigcup_{j \in [l]} \binom{C_j}{3},
\end{equation}
where $R_i$ and $C_j$ denote the row and column supports, respectively, of the matrix $\mathcal{Y}$ from~\eqref{matri}.

\begin{definition}\label{matroid gkl}
Recall from Theorem~\ref{theorem grid} and the discussion in Subsection~\ref{motivation} that there exists a matroid $M$ whose set of circuits is
\begin{equation}\label{description gkl}
\mathcal{C} = \min\left( \Delta \cup \textstyle\binom{[kl]}{4} \right),
\end{equation}
where 
$\min$ denotes the inclusion-minimal subsets, and satisfies $V_{\Delta} = V_{\mathcal{C}(M)}$. We denote this matroid by $G_{k,l}$. Moreover, we extend the definition of $G_{k,l}$ to the cases $k < 3$ or $l < 3$ by adopting the same construction.
\end{definition}

\begin{example}
In the cases $k=l=3$, and $k=3$, $l=4$, the associated matroids $G_{3,3}$ and $G_{3,4}$ coincide with the grid matroids depicted in Figure~\ref{3 times 3 and 3 times 4 grids}.
\end{example}

The decomposition of $V_{\Delta}$ then reduces to understanding the circuit variety of $G_{k,l}$. We now turn to studying the main properties of this matroid. In particular, the following lemma describes the lines of $G_{k,l}$ and shows that $G_{k,l}$ is a tame paving matroid.

\begin{lemma}\label{gkl is 2 paving}
$G_{k,l}$ is a tame paving matroid.
\end{lemma}

\begin{proof}
From the description of the circuits of $G_{k,l}$ given in Equation~\eqref{description gkl}, we see that every circuit has size either $3$ or $4$. Thus, $G_{k,l}$ is a paving matroid of rank three.

To verify that is tame, observe that a subset $S \subset [kl]$ of size $3$ is dependent in $G_{k,l}$ if and only if it is contained entirely within some row $R_i$ (for $i \in [k]$) or some column $C_j$ (for $j \in [l]$) of the matrix $\mathcal{Y}$. Therefore, the collection of lines of $G_{k,l}$ is
\begin{equation}\label{collection of dependent hyperplanes}
\mathcal{L}_{G_{k,l}} = \{ R_i : i \in [k],\ \size{R_i} \geq 3 \} \cup \{ C_j : j \in [l],\ \size{C_j} \geq 3 \}.
\end{equation}
Since each element $p \in [kl]$ belongs to at most one row and one column of $\mathcal{Y}$, it lies in at most two lines. Hence, $G_{k,l}$ is tame.
\end{proof}

We note that the proof of Lemma~\ref{gkl is 2 paving} can be easily adapted to the case where $s=t=n>3$.

\begin{example}
Figure~\ref{new figure 6} illustrates the matroids $G_{3,5}$ and $G_{4,5}$. The rows and columns of the underlying grid represent the lines (i.e., dependent hyperplanes) of these rank-three paving matroids.
\end{example}

Since $G_{k,l}$ is a tame paving matroid, we may apply Theorem~\ref{Theorem}. To do so, it is necessary to determine which partitions of $R \cup C$ are nice in the sense of Definition~\ref{nice}. Lemma~\ref{l<n} plays a key role in this identification. Before stating it, however, we establish an auxiliary result. Recall the notions of nilpotent and liftable paving matroids introduced in Subsection~\ref{subsection nilpotent and liftable}.

\begin{lemma}\label{liftable 2}
Let $M$ be an $n$-paving matroid on $[d]$, and let $p \in [d]$ be a non-coloop. Denote by $M' = M \setminus \{p\}$ the deletion of $p$ from $M$. Then the following statements hold:
\begin{itemize}
\item[{\rm (i)}] If $M'$ is liftable and $\deg(p) \leq 1$, then $M$ is liftable.
\item[{\rm (ii)}] If $M'$ is not liftable and $\deg(p) \geq 1$, then $M$ is not liftable.
\end{itemize}
\end{lemma}

\begin{proof}
(i) To prove that $M$ is liftable, we must show that for any tuple of vectors $\gamma =(\gamma_1,\ldots,\gamma_{d})$ in $\CC^n$ of rank $n-1$ spanning a hyperplane $H\subset \CC^{n}$, and for every $q \notin H$, there exists a non-degenerate lifting of $\gamma$ to a tuple in $V_{\mathcal{C}(M)}$.

According to the characterization of liftability given in~\textup{\cite[Section~3]{Fatemeh4}}, liftability is a closed condition, meaning it can be expressed as the vanishing of certain polynomials determined by the liftability matrix of $M$. We refer the reader to~\cite{Fatemeh4} for further details.

This allows us to reduce to the case where the vectors in $\gamma$ are in general position in $H$, that is, any subset of $n-1$ of them is linearly independent. Indeed, if the lifting exists for such generic configurations, it must also exist for arbitrary $\gamma$ by closedness.

Since $M' = M \setminus \{p\}$ is liftable, there exists a non-degenerate lifting $\widetilde{\gamma} = (\widetilde{\gamma}_r : r \in [d] \setminus \{p\}) \in V_{\mathcal{C}(M')}$ of the restriction of $\gamma$ to $[d] \setminus \{p\}$. We now distinguish two cases:

\medskip
\noindent
\textbf{Case 1.} $\deg(p) = 0$. In this case, $p$ is not contained in any dependent hyperplane. Therefore, we may define $\widetilde{\gamma}_p$ as any vector in $\CC^n$. The extended tuple $(\widetilde{\gamma}_1,\ldots, \widetilde{\gamma}_d)$ then lies in $V_{\mathcal{C}(M)}$.

\medskip
\noindent
\textbf{Case 2.} $\deg(p) = 1$, and let $\mathcal{L}_p = \{l\}$. Then $p$ appears in a unique dependent hyperplane $l$. Since the vectors in $\gamma$ are in general position in $H$, we have $\rank\{\gamma_{r}:r\in l\}=n-1$, which implies $\rank\{\widetilde{\gamma}_{r}:r\in l\}=n-1$.
We define $\widetilde{\gamma}_p$ as the projection of $\gamma_p$ from the vector $q$ onto the hyperplane spanned by the vectors $\{\widetilde{\gamma}_r : r \in l\}$. The extended tuple $(\widetilde{\gamma}_1,\ldots, \widetilde{\gamma}_d)$ then lies in $V_{\mathcal{C}(M)}$.

\medskip
(ii) Since $M'$ is not liftable, there exists a tuple of $d$ vectors $(\gamma_r : r \in [d] \setminus \{p\})$ in $\CC^n$ spanning a hyperplane $H\subset \CC^{n}$, and a vector $q \notin H$, such that this collection admits no non-degenerate lifting from $q$ to $V_{\mathcal{C}(M')}$. Let $l \in \mathcal{L}_p$ be a dependent hyperplane of $M$ containing $p$.

As liftability is a closed condition, we may perturb the vectors $\{\gamma_r\}$ slightly within $H$ while preserving non-liftability. In doing so, we can assume that $\rank(\gamma_l) = n - 1$.

Now extend this collection to all of $[d]$ by setting
\[
\widetilde{\gamma}_r = 
\begin{cases}
\gamma_r & \text{if } r \neq p, \\
v & \text{for some } v \in H \setminus \{0\} \  \text{if } r = p.
\end{cases}
\]
Assume, for contradiction, that $M$ is liftable. Then, there exists a lifting $(\tau_r : r \in [d]) \in V_{\mathcal{C}(M)}$ of $\widetilde{\gamma}$ from $q$, such that $\rank(\tau) = n$.

Since the original tuple $(\gamma_r)_{r \neq p}$ is not liftable from $q$, it follows that the lifted vectors $(\tau_r : r \in [d] \setminus \{p\})$ must all lie in a common hyperplane $H' \subset \CC^n$, which they span. From the assumption $\rank(\gamma_l) = n-1$, it follows that $\rank(\tau_l) = n - 1$, so $\text{span}(\tau_r : r \in l) = H'$. In particular, $\tau_p \in H'$, and thus all vectors $\{\tau_r : r \in [d]\}$ lie in $H'$, contradicting the assumption that $\rank(\tau) = n$.

This contradiction shows that $M$ cannot be liftable, completing the proof.
\end{proof}

\begin{lemma}\label{l<n}
The matroid $G_{k,l}$ is liftable if and only if either $k < 3$, $l < 3$, or $k = l = 3$.
\end{lemma}

\begin{proof}
We begin with the case $k < 3$. From the description of $\mathcal{L}_{G_{k,l}}$ given in Equation~\eqref{description gkl}, it follows that each element $p \in [kl]$ belongs to at most one line; namely, the row containing it. Thus, every element of $G_{k,l}$ has degree at most one, and so $G_{k,l}$ is nilpotent. By Proposition~\ref{new propo}, we conclude that $G_{k,l}$ is liftable. The case $l < 3$ is analogous. Now consider the case $k = l = 3$. By \textup{\cite[Example~3.10]{Fatemeh3}} the matroid $G_{3,3}$ is liftable.
%In this situation, the circuits of size $n$ in $G_{n,n}$ are precisely the $n$ rows and $n$ columns, giving a total of $2n$ such circuits. 
%The ground set has $n^2$ elements.
%We invoke \textup{\cite[Lemma~3.7]{Fatemeh4}}, which asserts that an $n$-paving matroid is liftable if the size of its ground set is at least the number of its size-$n$ circuits plus $n$. In our setting, $G_{n,n}$ has $n^2$ elements and exactly $2n$ circuits of size $n$, namely the rows and columns. Since
%\[
%n^2 \geq 2n + n
%\]
%for all $n \geq 3$, the hypothesis of the lemma is satisfied, and we conclude that $G_{n,n}$ is liftable.

It remains to show that $G_{k,l}$ is not liftable when $k, l \geq 3$ and $\max\{k,l\}\geq 4$. We argue by induction on $k + l$. The base case is $k + l = 7$. Without loss of generality, we may assume $k = 3$ and $l = 4$. The non-liftability of $G_{3,4}$ follows from \textup{\cite[Example~3.6]{Fatemeh3}}.

For the inductive step, suppose that $G_{k,l}$ is not liftable for all $k, l \geq 3$ with $k + l \leq m$. Let $k, l \geq 3$ with $k + l = m + 1$, and assume $k \geq l$. Then $G_{k,l}$ is obtained from $G_{k-1,l}$ by adding $l$ new elements, each of degree two (each of them lies in two lines which are its row and its column). Since $G_{k-1,l}$ is not liftable by the inductive hypothesis, it follows from Lemma~\ref{liftable 2}\,(ii) that $G_{k,l}$ is also not liftable. This completes the proof.
\end{proof}

We are now in a position to describe the nice partitions of $G_{k,l}$. For the next lemma, recall that if $k, l \geq 3$, then the collection $R \cup C$ of rows and columns of the matrix in~\eqref{matri} constitutes the set of lines of $G_{k,l}$, as specified in Equation~\eqref{collection of dependent hyperplanes}.

\begin{lemma}\label{nice RC}
Let $k, l \geq 3$, and let $\mathcal{Q} = \{Q_1, \ldots, Q_m\}$ be a partition of $R \cup C$, the set of the lines of $G_{k,l}$. Then $\mathcal{Q}$ is nice if and only if for each $i \in [m]$ with $\lvert Q_i \rvert \geq 2$, the following conditions hold:
\begin{itemize}
    \item[{\rm (i)}] $\lvert Q_i \cap R \rvert \geq 3$, $\lvert Q_i \cap C \rvert \geq 3$, and $\max\{\lvert Q_i \cap R \rvert, \lvert Q_i \cap C \rvert\} \geq 4$.
    \item[{\rm (ii)}] If $\lvert \mathcal{Q} \rvert > 1$, then neither $R \subseteq Q_i$ nor $C \subseteq Q_i$.
\end{itemize}
\end{lemma}

\begin{proof}
The partition $\mathcal{Q}$ is nice if and only if it satisfies conditions~(i) and~(ii) from Definition~\ref{nice}. By Definition~\ref{nice}, condition~(i) holds if and only if ${G_{k,l}}^{Q_{i}}$ is not liftable for each $i \in [m]$ such that $\size{Q_i} \geq 2$.

By Definition~\ref{submatroid of hyperplanes}, the matroid ${G_{k,l}}^{Q_i}$ is a paving matroid of rank three whose ground set consists of those entries of the matrix $\mathcal{Y}$ in~\eqref{matri} that lie in at least one row or column indexed by an element of $Q_i$, and whose lines are precisely those in $Q_i$.

From this, it follows that the elements of degree at least two in ${G_{k,l}}^{Q_i}$ are precisely the entries of $\mathcal{Y}$ whose row and column both lie in $Q_i$. Equivalently, these are the entries in the submatrix of $\mathcal{Y}$ determined by the rows indexed by $Q_i \cap R$ and columns indexed by $Q_i \cap C$. All remaining elements have degree one. Therefore, ${G_{k,l}}^{Q_i}$ is obtained from $G_{\size{Q_i \cap R}, \size{Q_i \cap C}}$ by adding elements of degree one. By Lemma~\ref{liftable 2}, ${G_{k,l}}^{Q_i}$ is not liftable if and only if $G_{\size{Q_i \cap R}, \size{Q_i \cap C}}$ is not liftable. By Lemma~\ref{l<n}, this occurs precisely when 
\[
\size{Q_i \cap R} \geq 3,\quad \size{Q_i \cap C} \geq 3,\quad \text{and} \quad \max\{\size{Q_i \cap R}, \size{Q_i \cap C}\} \geq 4,
\]
which is exactly condition~(i) in the statement of the lemma.

We now verify condition~(ii). Suppose that $R \subseteq Q_i$ or $C \subseteq Q_i$ for some $i \in [m]$. Then every entry of $\mathcal{Y}$ belongs to at least one line in $Q_i$, so the union $\cup_{H \in Q_i} H$ equals the full ground set $[kl]$. Hence, for condition~(ii) of Definition~\ref{nice} to hold, it must be that $Q_i = R \cup C$ in this case and $\size{\mathcal{Q}}=1$.

Now, suppose that $R \not\subseteq Q_i$ and $C \not\subseteq Q_i$ for all $i \in [m]$. We will show that Definition~\ref{nice}(ii) holds. We must show that for each $i\in [m]$ and each $l\in (R\cup C)\setminus Q_{i}$ it holds that
\[l\not \subset \textstyle\cup_{l^{\ast}\in Q_{i}}l^{\ast}.\] 
Fix any $i \in [m]$ and let $l \in (R\cup C) \setminus Q_i$. Without loss of generality, suppose $l \in R$. Since $C \not\subset Q_i$, there exists $l' \in C \setminus Q_i$. Then the entry of $\mathcal{Y}$ in row $l$ and column $l'$ does not lie in $\cup_{l^{\ast} \in Q_i} l^{\ast}$, because neither $l$ nor $l'$ belongs to $Q_i$. Hence, $l \not\subset \cup_{l^{\ast} \in Q_i} l^{\ast}$, verifying condition~(ii) of Definition~\ref{nice}.

We conclude that conditions~(i) and~(ii) from Definition~\ref{nice}, which characterize $\mathcal{Q}$ as nice, are equivalent to the corresponding conditions in the lemma. This completes the proof.
\end{proof}

With the characterization of the nice partitions of $G_{k,l}$ established in Lemma~\ref{nice RC}, we are now in a position to obtain the irredundant irreducible decomposition of $V_{\Delta}$ via Theorem~\ref{Theorem}.

\begin{theorem}\label{Theorem 2}
Let $s = t = n=3$. Then the variety $V_{\Delta}$ has the following irredundant irreducible decomposition:
\[
V_{\Delta} = \bigcup_{\mathcal{Q}} V_{G_{k,l}(\mathcal{Q})},
\]
where the union is taken over all partitions $\mathcal{Q} = \{Q_1, \ldots, Q_m\}$ of $R \cup C$ satisfying the following conditions for each $i \in [m]$ with $\size{Q_i} \geq 2$:
\begin{itemize}
    \item[{\rm (i)}] $\size{Q_i \cap R} \geq 3$, $\size{Q_i \cap C} \geq 3$, and $\max\{\size{Q_i \cap R}, \size{Q_i \cap C}\} \geq 4$.
    \item[{\rm (ii)}] If $\lvert \mathcal{Q} \rvert > 1$, then neither $R \subseteq Q_i$ nor $C \subseteq Q_i$.
\end{itemize}
\end{theorem}

\begin{proof}
By Lemma~\ref{gkl is 2 paving}, the matroid $G_{k,l}$ is a tame paving matroid, so the hypotheses of Theorem~\ref{Theorem} are satisfied for it. The result then follows immediately from Theorem~\ref{Theorem} together with the characterization of nice partitions in Lemma~\ref{nice RC}.
\end{proof}

\begin{remark}
We emphasize that most of the results in this section extend to the case $s=t=n>3$. 
What prevents us from establishing an analogue of Theorem~\ref{Theorem 2} for $n>3$ 
is the difficulty of determining when the matroids $G_{k,l}$ (defined for $n>3$) 
are liftable. For $n=3$, the base cases are provided by the liftability of $G_{3,3}$ 
and the non-liftability of $G_{3,4}$ established in \cite{Fatemeh3}. 
The situation for larger $n$ is unclear.
\end{remark}

We now specialize Theorem~\ref{Theorem 2} to the case $s = t = n = k=3$, yielding the following corollary.

\begin{corollary}\label{coro}
For the case where $s=t=n=k=3$ and $l>3$, the variety $V_{\Delta}$ has the following irredundant irreducible decomposition:
\[
V_{\Delta} = V_{G_{3,l}} \cup V_{U_{2,\, 3l}},
\]
where $U_{2,\, 3l}$ denotes the uniform matroid of rank $2$ on the ground set $[3l]$.
\end{corollary}

\begin{proof}
To determine the irreducible components in this case, it suffices to identify the partitions of $R \cup C$ satisfying conditions~(i) and~(ii) of Theorem~\ref{Theorem 2}. Let $\mathcal{Q} = \{Q_1, \ldots, Q_m\}$ be such a partition.

Fix any index $i \in [m]$ for which $\size{Q_i} \geq 2$. By condition~(i), we have $\size{Q_i \cap R} \geq 3$. Since $\size{R} = k = 3$, this implies that $R \subseteq Q_i$. Then, by condition~(ii), we must have $\size{\mathcal{Q}}=1$, and $\mathcal{Q} = \{R \cup C\}$.

The only other possibility is that each part of $\mathcal{Q}$ is a singleton, in which case no $Q_i$ has size greater than one, and both conditions are clearly satisfied.

Thus, the only two partitions satisfying the hypotheses of Theorem~\ref{Theorem 2} are:
\begin{itemize}
\item The trivial partition $\mathcal{Q}_1 = \{R \cup C\}$, for which $G_{3,l}(\mathcal{Q}_1) = U_{2, 3l}$;
\item The discrete partition $\mathcal{Q}_2 = \{\{R_1\}, \{R_{2}\}, \{R_3\}, \{C_1\}, \ldots, \{C_l\}\}$, for which $G_{3,l}(\mathcal{Q}_2) = G_{3,l}$.
\end{itemize}
The statement now follows directly from Theorem~\ref{Theorem 2}.
\end{proof}

\begin{remark}
Corollary~\ref{coro} establishes that when $s = t = n = k=3$ and $l>3$, the variety $V_{\Delta}$ has precisely two irreducible components: the matroid variety of $G_{3,l}$ and that of the uniform matroid $U_{2, 3l}$.

This generalizes a result from \cite{Fatemeh3}, where the authors verified the same decomposition in the specific case $s = t = n = k = 3$ and $l = 4$.
\end{remark}

\begin{figure}[H]
    \centering
    \includegraphics[width=0.5\textwidth, trim=0 0 0 0, clip]{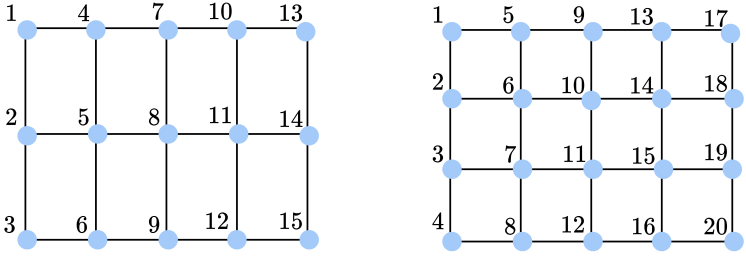}
    \caption{(Left) Matroid $G_{3,5}$; (Right) Matroid $G_{4,5}$.}
    \label{new figure 6}
\end{figure}

\begin{example}
Consider the rank-three matroid $G_{3,5}$ shown in Figure~\ref{new figure 6} (Left). By Corollary~\ref{coro}, we obtain the decomposition
\[
V_{\mathcal{C}(G_{3,5})} = V_{G_{3,5}} \cup V_{U_{2,15}}.
\]

Now consider the rank-three matroid $G_{4,5}$ illustrated in Figure~\ref{new figure 6} (Right). Applying Theorem~\ref{Theorem 2}, we find that
\begin{equation}\label{unionv}
V_{\mathcal{C}(G_{4,5})} = \bigcup_{\mathcal{Q}} V_{G_{4,5}(\mathcal{Q})},
\end{equation}
where the union is taken over all partitions $\mathcal{Q} = \{Q_{1}, \ldots, Q_{m}\}$ of $R \cup C$ satisfying the following for each $i \in [m]$ with $\size{Q_i} \geq 2$:
\begin{itemize}
\item[{\rm (i)}] $\size{Q_i \cap R} \geq 3$, $\size{Q_i \cap C} \geq 3$, and $\max\{\size{Q_i \cap R}, \size{Q_i \cap C}\} \geq 4$.
\item[{\rm (ii)}] If $\size{\mathcal{Q}} > 1$, then neither $R \subseteq Q_i$ nor $C \subseteq Q_i$.
\end{itemize}

The partitions $\mathcal{Q}$ satisfying these conditions are:
\begin{itemize}
\item The trivial partition $\{R \cup C\}$.
\item The discrete partition $\{\{R_1\}, \ldots, \{R_4\}, \{C_1\}, \ldots, \{C_5\}\}$.
\item All partitions of $R \cup C$ into three subsets: one singleton containing a column, one singleton containing a row, and one subset containing the remaining seven elements of $R\cup C$. Let $\mathcal{X}$ denote this collection of partitions.
\end{itemize}

Therefore, by \eqref{unionv}, we conclude that
\[
V_{\mathcal{C}(G_{4,5})} = V_{G_{4,5}} \cup V_{U_{2,20}} \cup \bigcup_{\mathcal{Q} \in \mathcal{X}} V_{G_{4,5}(\mathcal{Q})}.
\]

We conclude that, in this case, the variety $V_{\mathcal{C}(G_{4,5})}$ has exactly $22$ irreducible components.
\end{example}

%We conclude this section with the following natural numerical question.

%\begin{question}
%Given the description of the irreducible components of $V_{\Delta}$ in Theorem~\ref{Theorem 2}, can one compute, in purely combinatorial terms, the number of irreducible components of this variety?
%\end{question}

\section{Variety from \texorpdfstring{$n$}{n} lines in \texorpdfstring{$\mathbb{P}^{2}$}{P2}}\label{examples}

In this section, we apply Theorem~\ref{Theorem} to an interesting family of rank-three matroids. Specifically, we consider the matroids $L_n$, introduced in the following definition, which arise from a configuration $\mathcal{L}$ of $n$ lines in $\mathbb{P}^{2}$ together with their $\textstyle \binom{n}{2}$ pairwise intersection points. Our goal is to determine the irreducible components of the circuit variety associated with $L_n$.

%We first introduce the rank-three matroid $L_{n}$, defined by a set of $n$ lines and their $\textstyle \binom{n}{2}$ intersection points. %as described in \textup{\cite[Example~7.2]{Fatemeh4}}.

\begin{definition}\label{confi ln}
Let $n \geq 4$, and let $\mathcal{L}$ be a collection of $n$ lines in $\mathbb{P}^2$ in general position, meaning that no three lines of $\mathcal{L}$ meet at a common point. Let $\mathcal{P}$ denote the set of their $\textstyle \binom{n}{2}$ pairwise intersection points. The matroid $L_n$ is the rank-three paving matroid on the ground set $\mathcal{P}$, whose circuits of size three consist of all triples of distinct points lying on a common line of $\mathcal{L}$:
\[
\left\{ \{x, y, z\} \subset \mathcal{P} : \text{$x$, $y$, and $z$ lie on a common line in $\mathcal{L}$} \right\}.
\]

Equivalently, using Lemma~\ref{sub h}, $L_n$ can also be described in terms of its dependent hyperplanes (referred to as lines in rank three). The set of lines of $L_n$ is given by:
\[
\left\{ \{x \in \mathcal{P} : x \in \ell\} : \ell \in \mathcal{L} \right\}.
\]
For simplicity, we will identify the set of lines of the matroid $L_n$ with the collection $\mathcal{L}$ itself.
\end{definition}

It is important to note that the definition of the matroid $L_n$ does not depend on the specific choice of lines in $\mathcal{L}$.

\begin{example}
Figure~\ref{new figure 4} depicts the rank-three matroids $L_5$ and $L_6$. For instance, $L_5$ is defined on the ground set $[10]$, with the set of lines 
\[
\mathcal{L} = \{\{1,2,3,7\},\, \{1,5,6,10\},\, \{2,4,6,9\},\, \{3,4,5,8\},\, \{7,8,9,10\}\}.
\]
\end{example}

Observe that every element in the ground set of $L_n$ lies on exactly two lines. Hence, $L_n$ is a tame paving matroid, and the hypotheses of Theorem~\ref{Theorem} are satisfied. To apply this theorem, we must first determine which partitions of $\mathcal{L}$ are nice in the sense of Definition~\ref{nice}. For this purpose, we require the following lemma.

\begin{lemma}\label{liftable ln}
The matroid $L_n$ is not liftable.
\end{lemma}

\begin{proof}
We proceed by induction on $n\geq 4$. For the base case $n=4$, the matroid $L_4$ coincides with the quadrilateral set $QS$ (see Figure~\ref{fig:combined} (Right)), which was shown to be not liftable in~\textup{\cite[Example~3.6]{Fatemeh3}}.

For the inductive step, assume that $L_n$ is not liftable. The matroid $L_{n+1}$ is obtained from $L_n$ by adjoining $n$ elements of degree two. By Lemma~\ref{liftable 2}~(ii), it follows that $L_{n+1}$ is also not liftable, completing the induction.
\end{proof}

%For $n\leq 3$ the matroid $L_{n}$ is nilpotent. By applying Proposition~\ref{new propo} we conclude that $L_{n}$ is liftable in these cases.

We are now in position to describe the nice partitions of $L_{n}$.

\begin{lemma}\label{nice ln}
Let $\mathcal{Q} = \{Q_{1}, \ldots, Q_{k}\}$ be a partition of $\mathcal{L}$. Then $\mathcal{Q}$ is nice if and only if $\abs{Q_{i}} \notin \{2,3,n-1\}$ for every $i \in [k]$.
\end{lemma}

\begin{proof}
The partition $\mathcal{Q}$ is nice if and only if it satisfies conditions~(i) and~(ii) from Definition~\ref{nice}. By that definition, condition~(i) holds if and only if $L_{n}^{Q_{i}}$ is not liftable for every $i\in [k]$ such that $\abs{Q_{i}} \geq 2$.

By Definition~\ref{submatroid of hyperplanes}, the matroid $L_{n}^{Q_{i}}$ is a rank-three paving matroid whose ground set consists of the points of $\mathcal{P}$ lying on at least one line in $Q_{i}$, and whose lines are precisely the elements of $Q_{i}$.

If $\abs{Q_{i}} \leq 3$ for some $i \in [k]$, then it is easy to verify from Definition~\ref{def nilpotent} that $L_{n}^{Q_{i}}$ is nilpotent. In particular, Proposition~\ref{new propo} implies that $L_{n}^{Q_{i}}$ is liftable.

On the other hand, if $\abs{Q_{i}} \geq 4$, then the elements of degree at least two in $L_{n}^{Q_{i}}$ are precisely the points of $\mathcal{P}$ lying at the intersection of two lines in $Q_{i}$, while all other elements have degree one. Hence, $L_{n}^{Q_{i}}$ is obtained from $L_{\abs{Q_{i}}}$ by adding elements of degree one. Since $L_{\abs{Q_{i}}}$ is not liftable by Lemma~\ref{liftable ln}, Lemma~\ref{liftable 2} ensures that $L_{n}^{Q_{i}}$ is not liftable either.

Therefore, condition~(i) in Definition~\ref{nice} is equivalent to $\abs{Q_{i}} \notin \{2,3\}$ for every $i \in [k]$.

Now suppose that there exists an index $i \in [k]$ such that $\abs{Q_{i}} = n-1$, and let $l \in \mathcal{L}$ denote the unique line not contained in $Q_{i}$. Since every point on $l$ lies on some line in $Q_{i}$, it follows that $l \subset \cup_{l^{\ast} \in Q_{i}} l^{\ast}$,
and hence the second condition of Definition~\ref{nice} fails.

Conversely, assume that $\abs{Q_{i}} \leq n-2$ for every $i \in [k]$. We will show that condition~(ii) of Definition~\ref{nice} is satisfied in this case. That is, for each $i \in [k]$ and each $l \in \mathcal{L} \setminus Q_{i}$, we must verify that
\[
l \not\subseteq \bigcup_{l^{\ast} \in Q_{i}} l^{\ast}.
\]
Fix an arbitrary $i \in [k]$ and a line $l \notin Q_{i}$. Since $\abs{Q_{i}} \leq n-2$, there exists a line $l' \in \mathcal{L} \setminus (Q_{i} \cup \{l\})$. The intersection point of $l$ and $l'$ does not lie on any line in $Q_{i}$, as both $l$ and $l'$ do not belong to $Q_{i}$. Therefore, $l$ is not contained in the union of the lines in $Q_{i}$, and condition~(ii) holds.

We conclude that conditions~(i) and~(ii) in Definition~\ref{nice}, which characterize a nice partition, are equivalent to requiring that $\abs{Q_{i}} \notin \{2,3\}$ and $\abs{Q_{i}} \neq n-1$ for all $i \in [k]$. This completes the proof.
\end{proof}

With the characterization of nice partitions of $L_n$ given in Lemma~\ref{nice ln}, we are now in a position to describe the irreducible components of $V_{\mathcal{C}(L_n)}$ using Theorem~\ref{Theorem}.

\begin{theorem}\label{decomp ln}
The circuit variety of $L_{n}$ has the following irredundant irreducible decomposition:
\[V_{\mathcal{C}(L_{n})}=\bigcup_{\mathcal{Q}}V_{L_{n}(\mathcal{Q})},\]
where the union is taken over all partitions $\mathcal{Q}=\{Q_{1},\ldots,Q_{k}\}$ of $\mathcal{L}$ such that $\size{Q_{i}}\notin \{2,3,n-1\}$ for every $i\in [k]$.
\end{theorem}

\begin{proof}
The proof follows immediately from Theorem~\ref{Theorem} together with the characterization of nice partitions in Lemma~\ref{nice ln}.
\end{proof}

\begin{remark}
Theorem~\ref{decomp ln} provides the irreducible decomposition of $V_{\mathcal{C}(L_n)}$, extending a result from~\cite{Fatemeh3}, where the decomposition was established in the special case $n = 4$. This corresponds to the quadrilateral set configuration $QS$ shown in Figure~\ref{fig:combined} (Right).
\end{remark}

\begin{figure}[H]
    \centering
    \includegraphics[width=0.6\textwidth, trim=0 0 0 0, clip]{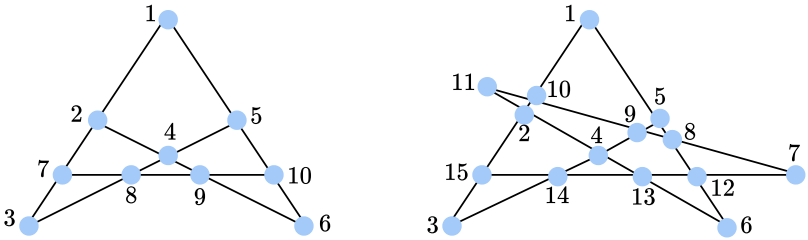}
    \caption{(Left) Matroid $L_{5}$; (Right) Matroid $L_{6}$.}
    \label{new figure 4}
\end{figure}

\begin{example}
Let us apply Theorem~\ref{decomp ln} to the matroids $L_5$ and $L_6$, illustrated in Figure~\ref{new figure 4}.

\smallskip

For $L_5$, we have the decomposition
\[
V_{\mathcal{C}(L_5)} = \bigcup_{\mathcal{Q}} V_{L_5(\mathcal{Q})},
\]
where the union is taken over all partitions $\mathcal{Q}$ of $\mathcal{L} = \{l_1, \dots, l_5\}$ such that $\abs{Q_i} \notin \{2,3,4\}$ for every part $Q_i$. There are exactly two such partitions:
\begin{itemize}
\item[{\rm (i)}] The discrete partition $\mathcal{Q}_1 = \big\{\{l_1\}, \dots, \{l_5\}\big\}$, for which $L_5(\mathcal{Q}_1) = L_5$;
\item[{\rm (ii)}] The trivial partition $\mathcal{Q}_2 = \big\{\{l_1, \dots, l_5\}\big\}$, for which $L_5(\mathcal{Q}_2) = U_{2,10}$.
\end{itemize}
Hence, the variety decomposes as
\[
V_{\mathcal{C}(L_5)} = V_{L_5} \cup V_{U_{2,10}}.
\]

\smallskip

For $L_6$, the decomposition
\[
V_{\mathcal{C}(L_6)} = \bigcup_{\mathcal{Q}} V_{L_6(\mathcal{Q})}
\]
is over all partitions $\mathcal{Q}$ of $\mathcal{L} = \{l_1, \dots, l_6\}$ such that $\abs{Q_i} \notin \{2,3,5\}$ for every part $Q_i$. These are:
\begin{itemize}
\item[{\rm (i)}] The discrete partition $\mathcal{Q}_1 = \big\{\{l_1\}, \dots, \{l_6\}\big\}$, for which $L_6(\mathcal{Q}_1) = L_6$;
\item[{\rm (ii)}] The trivial partition $\mathcal{Q}_2 = \big\{\{l_1, \dots, l_6\}\big\}$, for which $L_6(\mathcal{Q}_2) = U_{2,15}$;
\item[{\rm (iii)}] The 15 partitions of $\mathcal{L}$ with one part of size 4 and two parts of size 1. We denote this set of partitions by $\mathcal{S}$.
\end{itemize}
Thus, we obtain the decomposition
\[
V_{\mathcal{C}(L_6)} = V_{L_6} \cup V_{U_{2,15}} \cup \bigcup_{\mathcal{Q} \in \mathcal{S}} V_{L_6(\mathcal{Q})},
\]
from which we conclude that $V_{\mathcal{C}(L_6)}$ has exactly 17 irreducible components.
\end{example}

We conclude this section with a question for future research.

\begin{question}
Is it possible to obtain an analogue of Theorem~\ref{decomp ln} in higher dimensions by considering hyperplanes instead of lines?
\end{question}

\section{Number of irreducible components}\label{section: number of irreducible components}

The purpose of this section is to address the following two questions:

\begin{question}\label{question v delta}
How many irreducible components does the variety $V_{\Delta}$ from Section~\ref{sec 5} have?
\end{question}

\begin{question}\label{question v ln}
How many irreducible components does the variety $V_{\Ccal(L_{n})}$ from Section~\ref{examples} have?
\end{question}

Although we do not obtain explicit formulas for these numbers, we construct a generating function that encodes them. Our approach relies on classical results concerning generating functions of vector partitions, as we now explain. We note that deriving explicit formulas for the number of such components appears out of reach, as no closed expressions are known even for the number of vector partitions.

\subsection{Generating function for the number of vector partitions}

In this section, we recall some classical results on generating functions of vector partitions. 
For detailed proofs, we refer the reader to \cite{andrews1998theory,andrews2004integer,zbMATH03983158}. 
We begin by establishing the necessary notation and definitions.

\begin{definition}
A {\em vector partition} of $\mathbf{n}\in \ZZ_{\geq 0}^{d}$ is a multiset of vectors in $\ZZ_{\geq 0}^{d}\setminus \{0\}$ whose sum is $\fn$.
\end{definition}

\begin{notation}
We fix the following notation:
\vspace*{-.5em}

\hspace*{-1cm}\parbox{\linewidth}{
\begin{itemize}[label=$\blacktriangleright$]
\item For an integer vector $\mathbf{n}=(n_{1},\ldots,n_{d})\in \ZZ_{\geq 0}^{d}$, we denote \[\fn!=n_{1}!\cdots n_{d}!, \quad  x^{\fn}=x^{n_{1}}\cdots x^{n_{d}}\quad  \text{and} \quad [\fn]=[n_{1}]\amalg \cdots \amalg [n_{d}].\]
\item For a vector partition $\fl=(\mathbf{v}_{1},\ldots,\mathbf{v}_{d})$ of $\fn$, where $\mathbf{v}_{i}=(v_{i1},\ldots,v_{id})$, we denote 
\[\fl!=\mathbf{v}_{1}!\cdots \mathbf{v}_{d}!.\]
\item We denote the set of all vector partitions of a vector $\mathbf{n}\in \ZZ_{\geq 0}^{d}$ as $\mathcal{P}_{d}(\fn)$.\vs
\item We set $p_{d}(\fn)=\size{\mathcal{P}_{d}(\fn)}$.\vs
\item For a subset $S \subset \ZZ_{\geq 0}^{d}\setminus \{0\}$, we denote by $\mathcal{P}_{d}^{S}(\fn)$ the set of all vector partitions of $\fn$ in which none of its parts belongs to $S$, and set 
\[
p_{d}^{S}(\fn) := \bigl|\mathcal{P}_{d}^{S}(\fn)\bigr|.
\]
\end{itemize}}
\end{notation}

The next theorem gives the generating function for the number of vector partitions.

\begin{theorem}
The generating function of $p_{d}$ is
\[\sum_{\fn\in \ZZ_{\geq 0}^{d}}p_{d}(\fn)x^{\fn}=\prod_{\mathbf{m}\in \ZZ_{\geq 0}^{d}\setminus \{0\}}\frac{1}{1-x^{\mathbf{m}}}.\]
\end{theorem}

The intuition behind this formula is the following. Given a partition $\fl$ of $\fn$, we may decide independently, for each vector $\mathbf{m} \in \ZZ_{\geq 0}^{d}\setminus \{0\}$, how many times it appears as a part of $\fl$. Each occurrence of $\mathbf{m}$ contributes $\mathbf{m}$ to the total sum $\fn$. The generating series that encodes any number of repetitions of $\mathbf{m}$ is therefore 
\[
1 + x^{\mathbf{m}} + x^{2\mathbf{m}} + \cdots = \frac{1}{1-x^{\mathbf{m}}}.
\]
Multiplying these factors over all $\mathbf{m}$ yields the desired product formula. Using this interpretation, it is clear that
\begin{equation}\label{pds}\sum_{\fn\in \ZZ_{\geq 0}^{d}}p_{d}^{S}(\fn)x^{\fn}=\prod_{\mathbf{m}\in \ZZ_{\geq 0}^{d}\setminus (S\cup \{0\})}\frac{1}{1-x^{\mathbf{m}}}.\end{equation}

\begin{definition}\label{def: S-admissible}
Let $S \subset \ZZ_{\geq 0}^{d}\setminus \{0\}$ and let $\fn \in \ZZ_{\geq 0}^{d}$.
\begin{itemize}
\item A partition $\pi$ of $[\fn]=[n_{1}]\amalg \cdots \amalg [n_{d}]$ is called \emph{$S$-admissible} if there is no vector $\mathbf{s} = (s_1, \dots, s_d) \in S$ and block $B \in \pi$ such that $B$ contains exactly $s_i$ elements from $[n_i]$ for all $i=1,\dots,d$.
\item The set of all $S$-admissible partitions of $[\fn]$ is denoted by $\mathcal{Q}_S([\fn])$, and we write 
\[
q_S(\fn) := \bigl| \mathcal{Q}_S([\fn]) \bigr|.
\]
\end{itemize}
\end{definition}

Note that each partition $\pi$ of $[\fn]$ gives rise to a vector partition $\fl(\pi)$ of $\fn$, obtained by associating to every block $B\in \pi$ the vector whose $i$-th entry records the number of elements of $B$ coming from $[n_{i}]$, for $i=1,\ldots,d$. Clearly, $\pi$ is $S$-admissible if and only if $\fl(\pi)\in \mathcal{P}_{d}^{S}(\fn)$.  

Furthermore, a straightforward combinatorial argument shows that for each vector partition $\tau=(\tau_{1},\ldots,\tau_{d})$ of $\fn$, the number of set partitions $\pi$ of $[\fn]$ satisfying $\fl(\pi)=\tau$ is precisely
\[
\frac{\fn!}{\tau!\,\prod_{\mathbf{m}} a_{\mathbf{m}}!},
\]
where $a_{\mathbf{m}}$ denotes the multiplicity of the vector $\mathbf{m}$ in $\tau$.  

Incorporating these observations into the generating function for $p_{d}^{S}(\fn)$ from~\eqref{pds}, we deduce that the exponential generating function of $q_{S}$ is
\[
\sum_{\fn\in \ZZ_{\geq 0}^{d}} \frac{q_{S}(\fn)x^{\fn}}{\fn!}
= \prod_{\mathbf{m}\in \ZZ_{\geq 0}^{d}\setminus (S\cup \{0\})} 
\sum_{j\geq 0} \frac{(x^{\mathbf{m}})^{j}}{\mathbf{m}!^{\,j}\,j!}.
\]

Using the standard notation
\[
\exp(x)=\sum_{j\geq 0}\frac{x^{j}}{j!},
\]
we may present the formula more cleanly as follows.

\begin{proposition}\label{prop: exp gen function}
The exponential generating function of $q_{S}$ is
\[
\sum_{\fn\in \ZZ_{\geq 0}^{d}} \frac{q_{S}(\fn)x^{\fn}}{\fn!}
= \prod_{\mathbf{m}\in \ZZ_{\geq 0}^{d}\setminus (S\cup \{0\})}
\exp\!\left(\frac{x^{\mathbf{m}}}{\mathbf{m}!}\right).
\]
\end{proposition}

This result will play a central role in the following subsections, where we answer Questions~\ref{question v delta} and~\ref{question v ln}.

\subsection{Counting the irreducible components of $V_{\Delta}$}

In this subsection we address Question~\ref{question v delta} by deriving a generating function for the number of irreducible components of the variety $V_{\Delta}$ from Section~\ref{sec 5}. Recall that $V_{\Delta}$ is defined in terms of two parameters $k,l \geq 3$. Throughout this subsection, we assume $k,l > 4$. Indeed, by Corollary~\ref{coro}, the variety $V_{\Delta}$ has exactly two irreducible components when $(k,l)=(3,l)$ with $l>3$ or $(k,l)=(k,3)$ with $k>3$, while for $k=l=3$ the variety $V_{\Delta}$ is irreducible.

\begin{definition}
We denote by  
\[
c(k,l)=\#\{\text{irreducible components of } V_{\Delta}\},
\]
where $\Delta$ is the hypergraph defined in~\eqref{def delta} in terms of $k$ and $l$.
\end{definition}

By Theorem~\ref{Theorem 2}, the value $c(k,l)$ coincides with the number of partitions $\mathcal{Q} = \{Q_1,\ldots,Q_m\}$ of $R \cup C$ satisfying the following conditions for each $i \in [m]$ with $\lvert Q_i \rvert \geq 2$:
\begin{itemize}
    \item[{\rm (i)}] $\lvert Q_i \cap R \rvert \geq 3$, $\lvert Q_i \cap C \rvert \geq 3$, and $\max\{\lvert Q_i \cap R \rvert, \lvert Q_i \cap C \rvert\}\geq 4$,
    \item[{\rm (ii)}] If $\lvert \mathcal{Q} \rvert > 1$, then neither $R \subseteq Q_i$ nor $C \subseteq Q_i$.
\end{itemize}
Here, $R$ and $C$ denote the row and column sets of the matrix $\mathcal{Y}$ from~\eqref{matri}, and we may identify
\[
R \cup C \;\cong\; [k]\cup [\tilde{l}] \;=\; \{1,\ldots,k\}\cup \{\tilde{1},\ldots,\tilde{l}\}.
\]

\medskip

To phrase these conditions in the language of admissible partitions, let $k,l$ vary over the natural numbers.  

\begin{definition}
Define the subset $S\subset \mathbb{Z}_{\geq 0}^{2}\setminus\{0\}$ by
\begin{equation}\label{def S}
S \;:=\; \big(\{(a,b): 0\leq a<3 \ \text{or} \ 0\leq b<3 \}\cup \{(3,3)\}\big)\setminus \{(0,0),(1,0),(0,1)\}.
\end{equation}
\end{definition}

With this identification, $c(k,l)$ is precisely the number of partitions $\pi$ of $[k]\cup[\tilde{l}]$ such that:
\begin{itemize}
\item[{\rm (i)}] $\pi$ is $S$-admissible,
\item[{\rm (ii)}] If $\pi \neq \{[k]\cup [\tilde{l}]\}$, then no block $B \in \pi$ contains all of $[k]$ or all of $[\tilde{l}]$.
\end{itemize}

By Definition~\ref{def: S-admissible}, the number of partitions satisfying (i) is given by $q_S(k,l)$, whose exponential generating function is described in Proposition~\ref{prop: exp gen function}. It remains to subtract those partitions that satisfy (i) but fail condition (ii).  

\begin{lemma}\label{lem: bkl}
The number of partitions of $[k]\cup [\tilde{l}]$ that satisfy \textup{(i)} but not \textup{(ii)} is
\[
b(k,l)\;:=\;\frac{2^{k+1}-(k^{2}+k+4)\;+\;2^{l+1}-(l^{2}+l+4)}{2}.
\]
\end{lemma}

\begin{proof}
Suppose $\pi$ satisfies (i) but not (ii). Then some block $B\in \pi$ contains either all of $[k]$ or all of $[\tilde{l}]$. Assume the former, so that $B=[k]\cup A$ for some $A\subset [\tilde{l}]$. Condition (i) forces $|A|\geq 3$, while condition (ii) rules out $A=[\tilde{l}]$. Thus $\pi$ has the form
\[
\pi \;=\; \{B\}\cup\{\{i\}: i\in [\tilde{l}]\setminus A\},
\]
with $3\leq \size{A}\leq l-1$. Hence the number of possibilities for $A$ is
\[
\binom{l}{3}+\binom{l}{4}+\cdots+\binom{l}{l-1}
\;=\;\frac{2^{l+1}-(l^{2}+l+4)}{2}.
\]
The case where $[\tilde{l}]\subset B$ is analogous, giving the claimed expression for $b(k,l)$.
\end{proof}

\begin{lemma}\label{lem: exp b}
The exponential generating function of $b$ is 
\[\sum_{k,l\ge0}\frac{b(k,l)\,x^{k}y^{l}}{k!\,l!}=e^{2x+y}+e^{x+2y}-\tfrac{1}{2}\bigl(x^{2}+y^{2}+2x+2y+8\bigr)e^{x+y}.\]
\end{lemma}

\begin{proof}
Note that 
\begin{equation}
\begin{aligned}
\sum_{k,l\ge0}\frac{(2^{k}-\frac{k^{2}+k+4}{2})\,x^{k}y^{l}}{k!\,l!}&=(\sum_{k\geq 0}\frac{(2^{k}-\frac{k^{2}+k+4}{2})\,x^{k}}{k!})\sum_{l\geq 0}\frac{y^{l}}{l!}\\
&=(\sum_{k\geq 0}\frac{(2^{k}-\frac{k^{2}+k+4}{2})\,x^{k}}{k!})e^{y}\\
&=(e^{2x}-\frac{e^{x}}{2}(x^{2}+2x+4))e^{y}\\
&=e^{2x+y}-\frac{e^{x+y}}{2}(x^{2}+2x+4),
\end{aligned}
\end{equation}
where we are using $\sum_{k\geq 0}\tfrac{kx^{k}}{k!}=xe^{x}$ and $\sum_{k\geq 0}\tfrac{k^{2}x^{k}}{k!}=(x^{2}+x)e^{x}$. Performing the analogous argument for $l$, we obtained the desired formula.
\end{proof}

\medskip

Since $c(k,l)=q_S(k,l)-b(k,l)$, combining Proposition~\ref{prop: exp gen function} with Lemma~\ref{lem: exp b} yields the following generating function.

\begin{theorem}\label{theorem: count components 1}
The exponential generating function of $c$ is
\begin{equation}\label{eq:egf-c}
\begin{split}
\sum_{k,l\ge0}\frac{c(k,l)\,x^{k}y^{l}}{k!\,l!}
&=
\prod_{\substack{(m_{1},m_{2})\in\mathbb{Z}_{\ge0}^2\\\notin\,S\cup\{0\}}}
\exp\!\Big(\frac{x^{m_{1}}y^{m_{2}}}{m_{1}!\,m_{2}!}\Big) \\[4pt]
&\quad -\Bigl(e^{2x+y}+e^{x+2y}-\tfrac{1}{2}\bigl(x^{2}+y^{2}+2x+2y+8\bigr)e^{x+y}\Bigr),
\end{split}
\end{equation}
where $S$ is the set defined in~\eqref{def S}.
\end{theorem}

\begin{table}[H]
    \centering
    \begin{tabular}{c|c}
        $(k,l)$ & $c(k,l)$ \\\hline 
          (4,4) & 2  \\ 
          (4,5) & 22 \\ 
          (5,5) & 127 \\ 
           (4,6) & 86 \\ 
           (5,6) & 417 \\  
    \end{tabular}
    \caption{Number of irreducible components of $V_{\Delta}$ for small values of $k,l$.}
    \label{tab:my_label 2}
\end{table}

\subsection{Counting the irreducible components of $V_{\Ccal(L_{n})}$}

In this subsection we address Question~\ref{question v ln} by deriving a generating function for the number of irreducible components of the circuit variety $V_{\Ccal(L_{n})}$ from Section~\ref{examples}. 

\begin{definition}
For each $n \geq 0$, we define
\[
c(n)=\#\{\text{irreducible components of } V_{\Ccal(L_{n})}\}.
\]
\end{definition}

\begin{theorem}\label{thm: count components 2}
The exponential generating function of $c$ is
\[
\sum_{n\geq 0}\frac{c(n)x^{n}}{n!}
=\prod_{m\in \mathbb{Z}_{\geq 0}\setminus \{0,2,3\}}
\exp\!\left(\frac{x^{m}}{m!}\right) - x e^{x}.
\]
\end{theorem}

\begin{proof}
By Theorem~\ref{decomp ln}, the value $c(n)$ coincides with the number of partitions $\pi$ of $[n]$ such that no block $B \in \pi$ has size $2$, $3$, or $n-1$. By Definition~\ref{def: S-admissible}, the number of partitions avoiding block sizes $2$ and $3$ is given by $q_{\{2,3\}}(n)$. 

On the other hand, there are exactly $n$ partitions of $[n]$ containing a block of size $n-1$. Thus,
\[
c(n)=q_{\{2,3\}}(n)-n.
\]
Applying Proposition~\ref{prop: exp gen function}, we obtain
\[
\sum_{n\geq 0}\frac{c(n)x^{n}}{n!}
=\prod_{m\in \mathbb{Z}_{\geq 0}\setminus \{0,2,3\}}
\exp\!\left(\frac{x^{m}}{m!}\right)
-\sum_{n\geq 0}\frac{nx^{n}}{n!}.
\]
Since $\sum_{n\geq 0}\tfrac{nx^{n}}{n!}=xe^{x}$, the claim follows.
\end{proof}

\begin{table}[H]
    \centering
    \begin{tabular}{c|c}
        $n$ & $c(n)$ \\\hline 
          4 & 2  \\ 
          5 & 2 \\ 
          6 & 17 \\ 
          7 & 58 \\ 
          8 & 191 \\  
    \end{tabular}
    \caption{Number of irreducible components of $V_{\Ccal(L_{n})}$ for small values of $n$.}
    \label{tab:my_label}
\end{table}

\section{Proof of Theorem~\ref{Theorem}}\label{appendix}

In this section, we prove Theorem~\ref{Theorem}. Before proceeding, we introduce the notion of a \emph{perturbation}, understood as an arbitrarily small motion. This concept, which also appears in~\cite{hocsten2004ideals}, will play a central role in the arguments that follow.

\begin{definition}
Let $\gamma \subset \mathbb{C}^n$ be a collection of vectors, and let $X$ be a property of such collection. We say that a \emph{perturbation} can be applied to $\gamma$ to obtain a new collection satisfying $X$ if, for every $\epsilon > 0$, one can choose, for each $v \in \gamma$, a vector $\widetilde{v} \in \mathbb{C}^n$ such that $\lVert v - \widetilde{v} \rVert < \epsilon$, and the set $\{\widetilde{v} : v \in \gamma\}$ satisfies property $X$. When referring to a perturbation of a $k$-dimensional subspace $S \subset \mathbb{C}^n$, we fix a basis $\{v_1, \ldots, v_k\}$ of $S$ and apply a perturbation to this basis. The perturbed vectors then span a new $k$-dimensional subspace.
\end{definition}

Before turning to the proof of Theorem~\ref{Theorem}, we need to establish a series of auxiliary lemmas.

\begin{lemma}\label{rango}
Let $M$ be a tame paving matroid of rank $n$ with collection of dependent hyperplanes $\mathcal{L}$, and let $\mathcal{Q} = \{Q_1, \ldots, Q_k\}$ be a nice partition of $\mathcal{L}$. Denote by $\rank$ the rank function of $M(\mathcal{Q})$. Then the following hold:
\begin{itemize}
\item[{\rm (i)}] For every $l \in \mathcal{L}$, we have $\rank(l) = n - 1$.
\item[{\rm (ii)}] If $l_1, l_2 \in \mathcal{L}$ lie in different parts of the partition $\mathcal{Q}$, then $\rank(l_1 \cup l_2) = n$.
\end{itemize}
\end{lemma}

\begin{proof}
Recall from Definition~\ref{constr} that $M(\mathcal{Q})$ is the quasi-paving matroid obtained by applying the construction of Theorem~\ref{quasi} to the hypergraph $\mathcal{H} = \{H_1, \ldots, H_k\}$ and the integer $n$, where each $H_i$ is defined by $H_i = \cup_{l \in Q_i} l$.

\medskip
(i) Without loss of generality, suppose $l \in Q_1$. Since $\rank(H_1) \leq n - 1$, we have $\rank(l) \leq n - 1$. Given that $\mathcal{Q}$ is nice, condition~(ii) of Definition~\ref{nice} ensures that $l \not\subseteq H_j$ for every $j \neq 1$. Moreover, note that since $M$ is tame, each element of $l$ belongs to at most one other dependent hyperplane. To show that $\rank(l) = n - 1$, it suffices to exhibit a subset $B = \{x_1, \ldots, x_{n-1}\} \subset l$ that is independent in $M(\mathcal{Q})$. We consider two cases:

\medskip
\textbf{Case 1.} Suppose there exists $j \neq 1$ such that $|l \cap H_j| \geq n - 1$. Choose $B_1 \subseteq l \cap H_j$ with $|B_1| = n - 2$, and pick $x \in l \setminus H_j$. Then $B = B_1 \cup \{x\} \subset l$, and $B \not\subseteq H_j$. Since $B$ is not contained in the intersection of two sets of $\mathcal{H}$, it is independent.

\medskip
\textbf{Case 2.} Suppose $|l \cap H_j| \leq n - 2$ for every $j \neq 1$. Then any subset of $n - 1$ elements in $l$ is not contained in the intersection of two sets in $\mathcal{H}$, and is therefore independent.

\medskip
(ii) Assume without loss of generality that $l_1 \in Q_1$. From part~(i), there exists an independent set $\{x_1, \ldots, x_{n-1}\} \subset l_1$ in $M(\mathcal{Q})$. Since $\mathcal{Q}$ is nice, condition~(ii) of Definition~\ref{nice} implies that $l_2 \not\subseteq  H_1$, so we can choose $x \in l_2 \setminus H_1$. We claim that the set $\mathcal{B} = \{x, x_1, \ldots, x_{n-1}\}$ is independent. As $\mathcal{B} \subset l_1 \cup l_2$, this would imply $\rank(l_1 \cup l_2) = n$.

To prove the claim, we verify that $\mathcal{B}$ is neither contained in any $H_r$, and that no subset of size $n-1$ of $\mathcal{B}$ is included in the intersection of two sets $H_r \cap H_s$.

\medskip
(1) Suppose $\mathcal{B} \subset H_r$ for some $r \in [k]$. Then $\{x_1, \ldots, x_{n-1}\} \subset H_1 \cap H_r$, contradicting their independence, as they would form a circuit of Type~1.

\medskip
(2) Suppose there exist $r, s \in [k]$ and a subset $\{y_1, \ldots, y_{n-1}\} \subset \mathcal{B}$ such that $\{y_1, \ldots, y_{n-1}\} \subset H_r \cap H_s$. Since $x_1, \ldots, x_{n-1} \in H_1$ and the intersection of any three elements of $\mathcal{H}$ is empty, one of $r$ or $s$ must be $1$; assume without loss of generality that $r = 1$. Since $x \notin H_1$, the only possibility is $\{y_1, \ldots, y_{n-1}\} = \{x_1, \ldots, x_{n-1}\}$, implying again that they lie in $H_1 \cap H_s$, forming a Type~1 circuit, which again contradicts the independence of $\{x_{1},\ldots,x_{n-1}\}$.

\smallskip
Hence, $\mathcal{B}$ is independent, and $\rank(l_1 \cup l_2) = n$, as claimed.
\end{proof}

For the next lemma, we recall the definition of $\gamma_{l}$ from Definition~\ref{pav 2}.

\begin{lemma}\label{rango n-1}
Let $M$ be a tame paving matroid of rank $n$ on $[d]$. Then for any $\gamma \in V_{\mathcal{C}(M)}$, there exists a perturbation yielding $\widetilde{\gamma} \in V_{\mathcal{C}(M)}$ such that $\dim(\widetilde{\gamma}_{l}) = n-1$ for every $l \in \mathcal{L}$.
\end{lemma}

\begin{proof}
Let $l \in \mathcal{L}$ be any dependent hyperplane for which $\dim(\gamma_{l}) < n-1$. We will show that one can perturb $\gamma$ to obtain $\tau \in V_{\mathcal{C}(M)}$ with $\dim(\tau_{l}) = n-1$. Iterating this procedure over all such $l$ then yields the desired $\widetilde{\gamma}$.

Choose a subset $\{p_1, \ldots, p_{n-1}\} \subset l$, and let $l_1, \ldots, l_{n-1}$ denote the dependent hyperplanes, distinct from $l$, that contain the points $p_1, \ldots, p_{n-1}$, respectively. For simplicity, we assume that each $p_i$ lies in exactly two dependent hyperplanes; the general case can be treated similarly with minor adjustments.

Since $\gamma \in V_{\mathcal{C}(M)}$, for each $l' \in \mathcal{L}$ there exists a hyperplane $H_{l'} \subset \CC^n$ such that $\gamma_{l'} \subset H_{l'}$. In particular, as $\dim(\gamma_l) \leq n-2$, we can choose a hyperplane $H_l \subset \CC^n$ so that $\gamma_{l}\subset H_{l}$ and each intersection
\[
H_l \cap H_{l_1}, \ldots, H_l \cap H_{l_{n-1}}
\]
has dimension $n-2$, and not all of them coincide. Applying the Hall--Rado theorem~\cite[Theorem~2.1]{pym1970submodular}, we can choose vectors $v_i \in H_l \cap H_{l_i}$ such that $\{v_1, \ldots, v_{n-1}\}$ is linearly independent.

Since these vectors form a basis, there exists $\lambda \in \CC$ arbitrarily close to zero for which the vectors
\[
\gamma_{p_1} + \lambda v_1, \ldots, \gamma_{p_{n-1}} + \lambda v_{n-1}
\]
are linearly independent. Define the perturbed configuration $\tau$ by setting
\begin{itemize}
    \item $\tau_p = \gamma_p$ for all $p \notin \{p_1, \ldots, p_{n-1}\}$,
    \item $\tau_{p_i} = \gamma_{p_i} + \lambda v_i$ for $i = 1, \ldots, n-1$.
\end{itemize}
By construction, $\tau$ is a perturbation of $\gamma$ and satisfies $\tau_{l'} \subset H_{l'}$ for every $l' \in \mathcal{L}$. Hence, $\tau \in V_{\mathcal{C}(M)}$. Moreover, we have $\dim(\tau_l) = n-1$, completing the proof.
\end{proof}

Henceforth, given a matroid $M$ on $[d]$ and a tuple of vectors $\gamma = (\gamma_1, \ldots, \gamma_d)$, we write $\restr{\gamma}{S}$ to denote the restriction $(\gamma_p : p \in S)$ for any subset $S\subset [d]$. We also recall the notion of principal extension introduced in Definition~\ref{def principal extension}.

\begin{lemma}\label{rank F}
Let $M'$ be a matroid on the ground set $[d]$, and let $F$ be a flat of $M'$. Consider the principal extension $M = M' +_{F} a$ on $[d] \cup \{a\}$, and let $\gamma \in V_{\mathcal{C}(M)}$. Suppose that $\restr{\gamma}{[d]} \in V_{M'}$ and that $\rank(\restr{\gamma}{F}) = \rank(F)$. Then $\gamma \in V_{M}$.
\end{lemma}

\begin{proof}
Since $\restr{\gamma}{[d]} \in V_{M'}$, we can perturb it to obtain $\widetilde{\gamma} \in \Gamma_{M'}$. The assumption $\rank(\restr{\gamma}{F}) = \rank(F)$ ensures that the subspace $\text{span}(\widetilde{\gamma}_{p} : p \in F)$ is a perturbation of $\text{span}(\gamma_{p} : p \in F)$. Then, we apply a perturbation of $\gamma_{a}$ to obtain $\widetilde{\gamma}_a \in \text{span}(\widetilde{\gamma}_{p} : p \in F)$ such that $\widetilde{\gamma}_a \notin \text{span}(\widetilde{\gamma}_{p} : p \in F')$ for any proper flat $F \subsetneq F'$. This guarantees that $\widetilde{\gamma} \in \Gamma_M$, and hence $\gamma \in V_M$.
\end{proof}

We will also need the following lemma.

\begin{lemma}\label{prop no iguales}
Let $M$ be an $n$-quasi-paving matroid of rank $n$ on $[d]$, defined by the collection of subsets $\mathcal{H} = \{H_{1}, \ldots, H_{k}\}$. Suppose $\gamma \in V_{\mathcal{C}(M)}$ is such that for any distinct $i, j \in [k]$, the spans $\textup{span}(\gamma_{p} : p \in H_{i})$ and $\textup{span}(\gamma_{p} : p \in H_{j})$ are not equal to the same hyperplane in $\CC^{n}$. Then $\gamma \in V_{M}$.
\end{lemma}

\begin{proof}
We show that the vectors of $\gamma$ can be perturbed to obtain $\widetilde{\gamma} \in \Gamma_{M}$. By Proposition~\ref{extension}, there exists a sequence of elements of the ground set, $p_{1}, \ldots, p_{r}$, such that the submatroid $M' = M \setminus \{p_{1}, \ldots, p_{r}\}$ is a tame paving matroid of rank $n$, and $M$ is obtained from $M'$ via principal extensions at $p_{1}, \ldots, p_{r}$.

For each $i \in [k]$, let $H_i' = H_i \setminus \{p_{1}, \ldots, p_{r}\}$. The dependent hyperplanes of $M'$ are precisely those sets $H_i'$ for which $\abs{H_i'} \geq n$. By assumption, for each $i \in [k]$, there exists a hyperplane $Z_i \subset \CC^n$ such that $\text{span}(\gamma_p : p \in H_i) \subset Z_i$ and $Z_i \neq Z_j$ for all $i \neq j$.

After a suitable perturbation of the vectors of $\gamma$, we may further ensure that
\[
\rank\{\gamma_p : p \in H_i' \cap H_j'\} = n - 2,
\]
for any $i, j \in [k]$ such that $\abs{H_i \cap H_j} \geq n - 2$. Since
\[
\text{span}(\gamma_p : p \in H_i') \subset \text{span}(\gamma_p : p \in H_i) \subset Z_i,
\]
it follows that the spans $\text{span}(\gamma_p : p \in H_i')$ and $\text{span}(\gamma_p : p \in H_j')$ are distinct for $i \neq j$. Hence, by~\textup{\cite[Proposition~4.2]{Fatemeh4}}, we conclude that $\restr{\gamma}{M'} \in V_{M'}$. Finally, applying Lemma~\ref{rank F} to the sequence of principal extensions at $p_1, \ldots, p_r$, we deduce that $\gamma \in V_M$.
\end{proof}

We now proceed with the proof of Theorem~\ref{Theorem}. To this end, recall Definition~\ref{submatroid of hyperplanes}.

\begin{proof}[{\bf Proof of Theorem~\ref{Theorem}}]
Let $n$ denote the rank of $M$. To establish the theorem, it suffices to verify three aspects: the irreducibility of the components, the non-redundancy of the decomposition, and the equality
\begin{equation}\label{deco teo12}
V_{\mathcal{C}(M)} = \bigcup_{\mathcal{Q}} V_{M(\mathcal{Q})},
\end{equation}
where the union ranges over all nice partitions $\mathcal{Q}$ of the dependent hyperplanes of $M$.

\medskip
\noindent
\textbf{Irreducibility of the Components:} By Corollary~\ref{coro ir}, each variety $V_{M(\mathcal{Q})}$ is irreducible, as $M(\mathcal{Q})$ is quasi-paving for every nice partition $\mathcal{Q}$.

\medskip
\noindent
\textbf{Non-Redundancy of the Decomposition:}  
To establish that the decomposition is non-redundant, suppose, for the sake of contradiction, that there exist two distinct nice partitions $\mathcal{Q}_{1} \neq \mathcal{Q}_{2}$ satisfying $V_{M(\mathcal{Q}_{1})} \subset V_{M(\mathcal{Q}_{2})}$.
It follows that $\Gamma_{M(\mathcal{Q}_{1})} \subset V_{M(\mathcal{Q}_{2})}$.
Let $\mathcal{H} = \{H_{1}, \ldots, H_{k}\}$ and $\mathcal{H}' = \{H_{1}', \ldots, H_{k}'\}$ be the hypergraphs defining $M(\mathcal{Q}_{1})$ and $M(\mathcal{Q}_{2})$, respectively, as in Definition~\ref{constr}.  
We proceed by considering two cases:

\medskip
\noindent
\textbf{Case~1.} Suppose that $\mathcal{Q}_{2}$ is not a refinement of $\mathcal{Q}_{1}$. Then there exist distinct $l_{1}, l_{2} \in \mathcal{L}$ that lie in the same part of $\mathcal{Q}_{2}$ but belong to different parts of $\mathcal{Q}_{1}$. Since $l_{1}$ and $l_{2}$ belong to the same element of $\mathcal{Q}_{2}$, we have 
\[
\rank(l_{1} \cup l_{2}) \leq n-1 \quad \text{in } M(\mathcal{Q}_{2}).
\]
On the other hand, by Lemma~\ref{rango}~(ii), it holds that 
\[
\rank(l_{1} \cup l_{2}) = n \quad \text{in } M(\mathcal{Q}_{1}).
\]
By Corollary~\ref{coro ir}, the matroid $M(\mathcal{Q}_{1})$ is realizable, so we may choose $\gamma \in \Gamma_{M(\mathcal{Q}_{1})}$. As $\gamma$ realizes $M(\mathcal{Q}_{1})$, it satisfies $\gamma_{l_{1}} + \gamma_{l_{2}} = \CC^{n}$.
However, since $\Gamma_{M(\mathcal{Q}_{1})} \subset V_{M(\mathcal{Q}_{2})}$, we also have $\gamma \in V_{M(\mathcal{Q}_{2})}$. Given that $\rank(l_{1} \cup l_{2}) \leq n-1$ in $M(\mathcal{Q}_{2})$, it follows that $\dim(\gamma_{l_{1}} + \gamma_{l_{2}}) \leq n-1$,
which contradicts the fact that $\gamma_{l_{1}} + \gamma_{l_{2}} = \CC^{n}$. This contradiction completes the argument for Case~1.
%Clearly, $\gamma\not \in V_{M(\mathcal{Q}_{2})}$ since $M(\mathcal{Q}_{1})\not \geq M(\mathcal{Q}_{2})$, which is a contradiction.

%Therefore, there exists a set of points $D$ such that $D\in \mathcal{D}(M(\mathcal{Q}_{2}))$ and $D\notin \mathcal{D}(M(\mathcal{Q}_{1}))$.  Since $\gamma \in V_{M(\mathcal{Q}_{2})}$, we have that the set $\{\gamma_{p}:p\in D\}$ is dependent, which is a contradiction since $\gamma \in \Gamma_{M(\mathcal{Q}_{1})}$ and $D\notin \mathcal{D}(M(\mathcal{Q}_{1}))$.  

\medskip
\noindent
\textbf{Case~2.} Now suppose that $\mathcal{Q}_{2}$ is a refinement of $\mathcal{Q}_{1}$. Then there exists a part $\mathcal{R} \in \mathcal{Q}_{1}$ such that $\mathcal{R} = \sqcup_{i=1}^{j} \mathcal{S}_{i}$,
where each $\mathcal{S}_{i} \in \mathcal{Q}_{2}$ and $j > 1$. In particular, there exist $l_{1}, l_{2} \in \mathcal{R}$ that belong to distinct subsets in $\mathcal{Q}_{2}$. By Lemma~\ref{rango}~(ii), we then obtain $\rank\left( \cup_{l \in \mathcal{R}} l \right) = n$ in $M(\mathcal{Q}_{2})$. %Since $M^{\mathcal{R}}$ is not liftable, %by Lemma~\ref{liftable}, 
%we can find $\gamma \in \Gamma_{M(\mathcal{Q}_{1})}$ such that the restriction of $\gamma$ to $\cup_{l\in \mathcal{R}}l$ is not liftable relative to $M^{\mathcal{R}}$. 

\medskip
\noindent
\textbf{Claim 1.} There exists $\gamma \in \Gamma_{M(\mathcal{Q}_{1})}$ whose restriction to $\cup_{l \in \mathcal{R}} l$ is not liftable with respect to $M^{\mathcal{R}}$. 

\medskip
Since $\mathcal{Q}_{1}$ is a nice partition, the matroid $M^{\mathcal{R}}$ is not liftable. Define  
\[
S = \{ p \in \cup_{l \in \mathcal{R}} l \,: \text{$p$ lies in two distinct dependent hyperplanes of $\mathcal{R}$}\},\]
and let $M^{\mathcal{R}}|S$ denote the restriction of $M^{\mathcal{R}}$ to $S$. By Lemma~\ref{liftable 2}~(i), the non-liftability of $M^{\mathcal{R}}$ implies that $M^{\mathcal{R}}|S$ is also not liftable. Moreover, since the elements of $S$ lie in two dependent hyperplanes of $\mathcal{R}$, they all belong to a single subset of the hypergraph $\mathcal{H}$; in fact, to the same such subset. Thus, as $M(\mathcal{Q}_{1})$ is realizable, we may select $\gamma \in \Gamma_{M(\mathcal{Q}_{1})}$ such that its restriction $\restr{\gamma}{S}$ is not liftable with respect to $M^{\mathcal{R}}|S$. We claim that this $\gamma$ satisfies the desired properties.

Suppose, for contradiction, that the restriction of $\gamma$ to $\cup_{l \in \mathcal{R}} l$ is liftable with respect to $M^{\mathcal{R}}$. Then there exists a lifting $(\tau_{p} : p \in \cup_{l \in \mathcal{R}} l) \in V_{\mathcal{C}(M^{\mathcal{R}})}$ from a vector $q \in \mathbb{C}^{n}$ such that $\rank(\tau) = n$. However, since $M^{\mathcal{R}}|S$ is not liftable, the vectors $(\tau_{p} : p \in S)$ must lie in a common hyperplane $H \subset \mathbb{C}^{n}$.

Now consider an arbitrary element $x \in \cup_{l \in \mathcal{R}} l \setminus S$, and let $M^{\mathcal{R}}|(S \cup \{x\})$ denote the restriction of $M^{\mathcal{R}}$ to $S \cup \{x\}$. Again by Lemma~\ref{liftable 2}~(ii), this matroid is not liftable. Hence $x$ has degree at least one in $M^{\mathcal{R}}|(S \cup \{x\})$, meaning there exists a dependent hyperplane $l_{1}$ of $M^{\mathcal{R}}|(S \cup \{x\})$ such that $x \in l_{1}$. This $l_{1}$ corresponds uniquely to a dependent hyperplane $\widetilde{l_{1}}$ of $M^{\mathcal{R}}$, and since $\mathcal{R}$ is precisely the collection of dependent hyperplanes of $M^{\mathcal{R}}$, it follows that $\widetilde{l_{1}} \in \mathcal{R}$.

We now claim that $\rank(l_{1}) = n - 1$ in $M(\mathcal{Q}_{1})$. Suppose not. Then, by the construction of $M(\mathcal{Q}_{1})$, there would exist another dependent hyperplane $l_{2} \notin \mathcal{R}$ such that $S \cup \{x\} \subset l_{2}$. This would imply that every point in $l_{1}$ has degree one in $M^{\mathcal{R}}|S$, lying exclusively in $l_{1}$. This contradicts the non-liftability of $\restr{\gamma}{S}$ with respect to $M^{\mathcal{R}}|S$, as a lifting could then be achieved by lifting the vectors $\{\gamma_{p} : p \in l_{1}\}$ to a common hyperplane, while leaving the remaining elements in $S \setminus l_{1}$ unchanged. We conclude that $\rank(l_{1}) = n - 1$ in $M(\mathcal{Q}_{1})$.

Since $\gamma$ realizes $M(\mathcal{Q}_{1})$ and $\rank(l_{1}) = n - 1$ in this matroid, it follows that $\rank\{\gamma_{p} : p \in l_{1}\} = n - 1$. Hence, $\rank\{\tau_{p} : p \in l_{1}\} = n - 1$ as well, and $\mathrm{span} (\tau_{p} : p \in l_{1}) = H$. In particular, since $x \in l_{1}$, we have $\tau_{x} \in H$. But $x$ was arbitrary in $\cup_{l \in \mathcal{R}} l \setminus S$, so we deduce that $\{\tau_{p} : p \in \cup_{l \in \mathcal{R}} l\} \subset H$, contradicting the fact that $\rank(\tau) = n$. This establishes the claim.

\medskip

%\begin{lemma}\label{liftable}
%Let $M$ be a matroid of rank $n$ and %over the ground set $[d]$. 
%let $\gamma\in V_{\mathcal{C}(M)}$ with $\rank(\gamma)=n-1$. Suppose that there exists $N\geq M$ with $\rank(N)=n$ and $\gamma \in V_{N}$. Then, for any vector $q\in \CC^{n}$ outside $\gamma$, the vectors of $\gamma$ are liftable from $q$.
%\end{lemma}

%\begin{proof}
%Since $\gamma \in V_{N}$, we know that for any $\epsilon>0$ there exists $\gamma(\epsilon)\in \Gamma_{N}$ with $\size{\gamma-\gamma(\epsilon)}<\epsilon $. Since $\gamma(\epsilon) \in V_{\mathcal{C}(M)}$ and $\rank(\gamma(\epsilon))=n$, we have that $\gamma(\epsilon)$ is liftable from $q$. Then, by \textup{\cite[Lemma~4.27]{Fatemeh5}}, we have that the $(\size{M}-n+1)$-minors of the matrix $\mathcal{M}_{q}^{\gamma(\epsilon)}(M)$ are zero. Taking limit $\epsilon \rightarrow 0$, we deduce that the $(\size{M}-n+1)$-minors of the matrix   $\mathcal{M}_{q}^{\gamma}(M)$ are zero. Applying again \textup{\cite[Lemma~4.27]{Fatemeh5}}, we have that the vectors of $\gamma$ are liftable from $q$.
%\end{proof}

Let $\gamma \in \Gamma_{M(\mathcal{Q}_{1})}$ be as in Claim~1. Since $\Gamma_{M(\mathcal{Q}_{1})} \subset V_{M(\mathcal{Q}_{2})}$, it follows that $\gamma \in V_{M(\mathcal{Q}_{2})}$. Moreover, as $\rank\left(\cup_{l \in \mathcal{R}} l\right) = n$ in $M(\mathcal{Q}_{2})$, the restriction of any realization of $M(\mathcal{Q}_{2})$ to $\cup_{l \in \mathcal{R}} l$ is liftable relative to $M^{\mathcal{R}}$. Since liftability is a closed condition and $\gamma \in V_{M(\mathcal{Q}_{2})}$, it follows that the restriction of $\gamma$ to $\cup_{l \in \mathcal{R}} l$ must also be liftable relative to $M^{\mathcal{R}}$, contradicting the defining property of $\gamma$. This contradiction confirms that the decomposition is non-redundant.

%Since $\rank_{M(\mathcal{Q}_{2})}(\cup_{l\in \mathcal{R}}l)=n$ and $M(\mathcal{Q}_{2})\geq M$, and $\restr{\gamma}{M^{\mathcal{R}}}\in V_{M(\mathcal{Q}_{2})[M^{\mathcal{R}}]}$, we have by Lemma~\ref{liftable}, we have that $\restr{\gamma}{M^{\mathcal{R}}}$ is liftable relative to $M^{\mathcal{R}}$. We arrived at a contradiction, hence the decomposition is non-redundant. 

\medskip
\noindent
\textbf{Equality in \eqref{deco teo12}.} We now prove the equality in \eqref{deco teo12}. By Lemma~\ref{mayor igual}, we have $M(\mathcal{S}) \geq M$ for any partition $\mathcal{S}$, which implies that $V_{M(\mathcal{S})} \subset V_{\mathcal{C}(M)}$. This yields the inclusion $\supset$. 

Now let $\gamma \in V_{\mathcal{C}(M)}$. Our goal is to show that $\gamma$ lies in the variety on the right-hand side of \eqref{deco teo12}. By Lemma~\ref{rango n-1}, we may assume that $\rank(\gamma_{l}) = n-1$ for every $l \in \mathcal{L}$. 

Applying \textup{\cite[Proposition~4.26]{Fatemeh4}}, we can perturb $\gamma$ to obtain a tuple $\tau \in V_{\mathcal{C}(M)}$ satisfying the following condition:

\begin{itemize}
\item Let $\mathcal{Q} = \{Q_{1}, \ldots, Q_{k}\}$ be the partition of $\mathcal{L}$ such that $l_1, l_2 \in \mathcal{L}$ lie in the same subset of $\mathcal{Q}$ if and only if $\tau_{l_1} = \tau_{l_2}$. Then, for every $i \in [k]$ such that $\size{Q_i} > 2$, the restriction of $\tau$ to $\cup_{l \in Q_i} l$ is not liftable relative to $M^{Q_i}$.
\end{itemize}

\noindent
\textbf{Claim 2.} The partition $\mathcal{Q}$ is nice.

\medskip
Since the restriction of $\tau$ to $\cup_{l \in Q_i} l$ is not liftable, it follows that $M^{Q_i}$ is not liftable for every $i \in [k]$. Hence, the first condition of Definition~\ref{nice} is satisfied.

Now, suppose there exists a line $l \notin Q_i$ such that $l \subset \cup_{l' \in Q_i} l'$. Then $\tau_l = \tau_{l'}$ for every $l' \in Q_i$, which would imply that $l \in Q_i$, contradicting the assumption that $l \notin Q_i$. Thus, the second condition of Definition~\ref{nice} also holds. We conclude that $\mathcal{Q}$ is nice.

\medskip
\noindent
\textbf{Claim 3.} We have $\tau \in V_{\mathcal{C}(M(\mathcal{Q}))}$, and the hyperplanes
\[
\operatorname{span}(\tau_p : p \in \cup_{l \in Q_i} l) \quad \text{and} \quad \operatorname{span}(\tau_p : p \in \cup_{l \in Q_j} l)
\]
are distinct for any $i \neq j \in [k]$.

\medskip
The second part follows from the fact that the hyperplanes in $Q_i$ and $Q_j$ belong to distinct parts of the partition $\mathcal{Q}$, hence their associated spans differ. To show that $\tau \in V_{\mathcal{C}(M(\mathcal{Q}))}$, two conditions need to be checked:

\medskip
\noindent
\textit{(i)} Let $C$ be a set of $n-1$ points in $(\cup_{l \in Q_i} l) \cap (\cup_{l \in Q_j} l)$. Then,
\[
\{\tau_p : p \in C\} \subset \operatorname{span}(\tau_p : p \in \cup_{l \in Q_i} l) \cap \operatorname{span}(\tau_p : p \in \cup_{l \in Q_j} l),
\]
which implies
\[
\rank\{\tau_p : p \in C\} \leq n - 2,
\]
and thus $\{\tau_p : p \in C\}$ is dependent.

\medskip
\noindent
\textit{(ii)} Let $C$ be a set of $n$ points in $\cup_{l \in Q_i} l$. Since
\[
\rank\{\tau_p : p \in \cup_{l \in Q_i} l\} = n - 1,
\]
the set $\{\tau_p : p \in C\}$ is also dependent.

\medskip
These two conditions verify that $\tau \in V_{\mathcal{C}(M(\mathcal{Q}))}$. Finally, by Lemma~\ref{prop no iguales}, we conclude that $\tau \in V_{M(\mathcal{Q})}$. Since $\tau$ is a perturbation of $\gamma$, this completes the proof of the reverse inclusion.
\end{proof}

%\section{Conditional independence models}

\vspace{-3mm}

%\bibliographystyle{plain}

%\bibliography{Citation}

\printbibliography

\medskip
{\footnotesize\noindent {\bf Author's addresses}
\medskip

\noindent{Emiliano Liwski, 
KU Leuven}\hfill {\tt  emiliano.liwski@kuleuven.be}

\end{document}
.

\begin{lemma}\label{suf 3-4}
Let $S=\{a_{1},\ldots,a_{k}\}$ be a multiset of numbers of $[n]$. 
\begin{itemize}
    \item[{\rm (i)}] Assume $S_{[k]}=S_{[k]}^{n}>0$. Let $V_{1},\ldots V_{k}$ be subspaces of $\CC^{n}$ such that for every $\emptyset \subsetneq L\subset [k]$ we have 
$\dim(V_{L})\leq S_{L}$.
Then, there exists subspaces $W_{1},\ldots,W_{k}$ such that 
\[
\dim(W_{i})=a_{i},\ V_{i}\subset W_{i},\ \text{ {\rm for every} $i$\quad {\rm and } }\ \dim(W_{[k]})\ = S_{[k]}.
\]

 \item[{\rm (ii)}] Let $C_{1},\ldots,C_{d}$ be non-empty subsets of $[k]$ with $d\leq n$. Suppose that
\begin{itemize}
\item For each $L\subset [k]$ 
%included in some $C_{j}$ 
there are at most $S_{L}$ numbers of $[d]$ such that $L\subset C_{i}$.
\item $S_{C_{i}}=S_{C_{i}}^{n}>0$ for every $i\in [d]$. 
\end{itemize}
Then there exist subspaces $W_{1},\ldots ,W_{k}$ of $\CC^{n}$ such that 
\begin{itemize}
    \item[{\rm (a)}] $\dim(W_{i})=a_{i},\ \text{{\rm for every $i\in [k]$,}} \quad \dim(W_{C_{i}})=S_{C_{i}}\ \text{{\rm for every $i\in [d]$}} \text{ and }$

    \item[{\rm (b)}] the subspaces
    $W_{C_{1}},\ldots,W_{C_{d}}$
have the \textit{choice property}.
\end{itemize}
Moreover, any collection $V_{1},\ldots ,V_{k}$ of subspaces of $\CC^{n}$ with $\dim(V_{i})=a_{i}$, can be infinitesimally perturbed to obtain subspaces $\widetilde{V}_{1},\ldots,\widetilde{V}_{k}$ satisfying {\rm (a)} and {\rm (b)}. %with $\dim (\widetilde{V}_{C_{j}})=S_{C_{j}}$ for all $j\in [d]$ such that the subspaces $\widetilde{V}_{C_{1}},\ldots,\widetilde{V}_{C_{d}}$ have the \textit{choice property}.
\end{itemize}
\end{lemma}